\documentclass[final]{amsart}
\usepackage{amsmath}
\usepackage{amssymb}
\newtheorem{theorem}{Theorem}[section]
\newtheorem{lemma}[theorem]{Lemma}
\newtheorem{corollary}[theorem]{Corollary}
\newtheorem{prop}[theorem]{Proposition}

\newtheorem{remark}[theorem]{Remark}
\newtheorem{definition}[theorem]{Definition}
\newtheorem{example}[theorem]{Example}

\numberwithin{equation}{section}

\providecommand{\AMS}{$\mathcal{A}$\kern-.1667em%
\lower.25em\hbox{$\mathcal{M}$}\kern-.125em$\mathcal{S}$}
\begin{document}
\title{On approximate orthogonality and symmetry of operators in semi-Hilbertian structure}

\author[Jeet Sen,  Debmalya Sain and Kallol Paul]{Jeet Sen, Debmalya Sain and Kallol Paul }

\newcommand{\acr}{\newline\indent}

\address[Sen]{Department of Mathematics\\ Jadavpur University\\ Kolkata 700032\\ West Bengal\\ INDIA}
\email{senet.jeet@gmail.com}

\address[Sain]{Department of Mathematics\\ Indian Institute of Science\\ Bengaluru 560012\\ Karnataka\\ INDIA}
\email{saindebmalya@gmail.com}

\address[Paul]{Department of Mathematics\\ Jadavpur University\\ Kolkata 700032\\ West Bengal\\ INDIA}
\email{kalloldada@gmail.com}

\thanks{ The research of Jeet Sen is supported by CSIR, Govt. of India. The research of Prof. Paul  is supported by project MATRICS(MTR/2017/000059)  of DST, Govt. of India.}

\maketitle
\begin{abstract}
The purpose of the article is to generalize the concept of approximate Birkhoff-James orthogonality, in the semi-Hilbertian structure. Given a positive operator $ A $ on a Hilbert space $ \mathbb{H}, $ we define $ (\epsilon,A)- $approximate orthogonality and $ (\epsilon,A)- $approximate orthogonality in the sense of Chmieli$\acute{n}$ski and establish a relation between them. We also characterize $ (\epsilon,A)- $approximate orthogonality in the sense of Chmieli$\acute{n}$ski for $A$-bounded and $A$-bounded compact operators. We further generalize the concept of right symmetric and left symmetric operators on a Hilbert space. The utility of these notions are illustrated by extending some of the previous results obtained by various authors in the setting of Hilbert spaces.

\end{abstract}

\maketitle
\section{Introduction}
In recent times, the geometry of operator spaces has been investigated by various authors using Birkhoff-James orthogonality techniques \cite{BS, GSP, S, sp, spm, turn, t}. The notion of approximate Birkhoff-James orthogonality \cite{C, CSW, dskpam} allows us to generalize some of the results presented in these articles and also to study some related perturbation and stability questions related closely to the geometric and analytic structure of Banach (Hilbert) spaces. Our aim in the present article is to study the concept of approximate Birkhoff-James orthogonality in the semi-Hilbertian structure, as introduced in \cite{acg, acg2}. Let us now introduce some relevant notations and terminologies which will be used throughout this article.\\

 We use the symbols $ \mathbb{H}$ and $ \mathbb{X}$ to denote a Hilbert space and a Banach space, respectively. Unless mentioned specifically, we work with both real and complex Hilbert spaces. The scalar field is denoted by $ \mathbb{K}(=\mathbb{R} ~\text{or}~ \mathbb{C}).$   The underlying inner product and the corresponding norm on $ \mathbb{H} $ are denoted by $ \langle~,~ \rangle $ and $ \|\cdot\|, $ respectively. Let  $\mathbb{L}(\mathbb{H})(\mathbb{K}(\mathbb{H}))$ denote the Banach space of all bounded (compact) linear operators on $ \mathbb{H} $, endowed with the usual operator norm.  $A \in \mathbb{L}(\mathbb{H})$ is said to be a positive operator if $A = A^*$ and $\langle
 Ax,x\rangle  \geq 0 $ for all $x \in \mathbb{H}$. A positive operator $ A $ is said to be positive definite if  $\langle Ax,x\rangle > 0$ for all $ x \ ( \neq \theta)\in \mathbb{H} $. It is well known \cite{acg} that  any positive operator $A \in \mathbb{L}(\mathbb{H})$ induces a positive semi-definite sesquilinear form $\langle ~,~\rangle _A$ on $\mathbb{H},$ given by $\langle x,y\rangle _A = \langle Ax,y\rangle ,$ where $x , y \in \mathbb{H}.$ It is easy to see that $\langle ~,~\rangle _A$ induces a semi-norm $ \|\cdot\|_A $ on $ \mathbb{H}, $ given by $ \| x \|_A = \sqrt{\langle Ax,x\rangle}.$ Henceforth we reserve the symbol $A$ for a positive operator on $\mathbb{H}.$ The null space and the range space of $A$ is denoted by $N(A)$ and  $R(A),$ respectively. The symbol $I$ is used to denote the identity operator on the concerned space.

 In a normed space $\mathcal{N}$, an element $x \in \mathcal{N}$ is said to be \textit{Birkhoff-James orthogonal} \cite{B,J} to another element $y \in \mathcal{N}$, written as $x \bot_{B} y$ if $\|x +\lambda y\| \geq \|x\|$ for all $\lambda \in \mathbb{K}.$ If the normed space $\mathcal{N}$ is a Hilbert space, then the Birkhoff-James orthogonality is equivalent to usual inner product orthogonality $ \langle x, y\rangle =0.$ In \cite{C}, the author introduced the following definition of an approximate version of Birkhoff-James orthogonality:
 \begin{definition}\label{approximate Birkhoff orthogonal}
  Let $\mathcal{N}$ be a normed space and  $\epsilon \in [0,1).$ An element $x \in \mathcal{N}$ is said to be approximate Birkhoff-James orthogonal to another element $y \in \mathcal{N}$, denoted by $x \bot^{\epsilon}_{B} y,$ if $ \|x + \lambda y \|^2 \geq \|x\|^2 -2\epsilon \|x\| \|\lambda y\|$ for all $\lambda \in \mathbb{K}.$
 \end{definition}
 Clearly, if $\epsilon =0,$ the above definition coincides with the definition of Birkhoff-James orthogonality in a normed space $ \mathcal{N} $. In the Hilbert space setting, Definition \ref{approximate Birkhoff orthogonal} coincides with the usual definition of approximate $\epsilon-$orthogonality. We note that an element $x \in \mathbb{H}$ is said to be \textit{approximate $ \epsilon- $orthogonal} to another element $y \in \mathbb{H}$, denoted by $x \bot^{\epsilon} y,$ if $|\langle x,y \rangle| \leq \epsilon \|x\| \|y\|,$ where $\epsilon \in [0,1)$ (see \cite{C}). It should be mentioned here that there exists another standard notion of approximate Birkhoff-James orthogonality in normed spaces (see \cite{SSD}), which is not dealt with in this work.\

 We now focus on the semi-Hilbertian structure on $ \mathbb{H}$ induced by the positive operator $A$. An operator $ T \in \mathbb{L}(\mathbb{H}) $ is said to be $ A- $bounded if there exists $ c > 0$ such that  $\|Tx\|_A \leq c\|x\|_A ~\forall
 x\in \mathbb{H} .$
 Let $B_{A^{1/2}}(\mathbb{H} )$ denote the collection of all $A$-bounded operators, i.e., $   B_{A^{1/2}}(\mathbb{H} )= \big\{ T \in \mathbb{L}(\mathbb{H}) :~ \exists~ c > 0 ~\text{such that}~\|Tx\|_A \leq c\|x\|_A ~\forall
 x\in \mathbb{H} \big\}. $ The $ A- $norm of $ T \in B_{A^{1/2}}(\mathbb{H}) $ is given by:
 \[ \|T\|_A = \sup_{x\in \mathbb{H}, \|x\|_A =1}  \|Tx\|_A = \sup \left\lbrace |\langle Tx,y\rangle _A|: x,y \in \mathbb{H}, \|x\|_A= \|y\|_A =1\right\rbrace. \]  Let us now recall some relevant definitions from  \cite{acg} and \cite{z}.

\begin{definition}\label{A ortho}
 Let  $ \mathbb{H} $ be a Hilbert space. An element $x \in \mathbb{H}$ is said to be $A$-orthogonal to an element $y \in \mathbb{H},$ denoted by $x \bot_A y,$ if $\langle x, y\rangle_A = 0.$
\end{definition}
Clearly, in a Hilbert space $\mathbb{H}$, for a positive operator $A,$ $x \bot_A y ~\Leftrightarrow ~ \|x+ \lambda y\|_A \geq \|x\|_A$ for all $\lambda \in \mathbb{K}$. Further note that if $A=I$, then the above definition coincides with the usual notion of orthogonality in Hilbert spaces, which in turn is equivalent to Birkhoff-James orthogonality.

\begin{definition}\label{A BJ ortho}Let $\mathbb{H}$ be a Hilbert space.
$T \in B_{A^{1/2}}(\mathbb{H})$ is said to be $A$-Birkhoff-James orthogonal to $S \in B_{A^{1/2}}(\mathbb{H}),$ denoted by
$T\bot_A^B S, $ if $\|T+\gamma S\|_A \geq \|T\|_A$ for all $\gamma \in \mathbb{K}.$
\end{definition}

Note that the above definition gives a generalization of the Birkhoff-James orthogonality of bounded linear operators on a Hilbert space. We also make use of the following notations:\\
For a positive operator  $A \in \mathbb{L}(\mathbb{H}),$ $B_{\mathbb{H}(A)}$ and $S_{\mathbb{H}(A)}$ denote the $A$-unit ball
and the $A$-unit sphere of $ \mathbb{H}, $ respectively, i.e., $B_{\mathbb{H}(A)} = \left\lbrace x \in \mathbb{H} : \|x\|_A \leq 1\right\rbrace $ and  $S_{\mathbb{H}(A)} = \left\lbrace x \in \mathbb{H} :
\|x\|_A = 1\right\rbrace $. For any $T \in B_{A^{1/2}}(\mathbb{H}),$ the $A$-norm attainment set $M^T_A$ of $ T $  was considered in \cite{z}:
\[ M^T_A = \left\lbrace x \in \mathbb{H} : \|x\|_A = 1,~\|Tx\|_A = \|T\|_A \right\rbrace. \]\\
Let us now define  $(\epsilon,A)-$approximate orthogonality in the semi-Hilbertian structure.
\begin{definition}
  Let $\mathbb{H}$ be a Hilbert space and  $\epsilon \in [0,1).$ An element $x \in \mathbb{H}$ is said to be $(\epsilon,A)$-approximate orthogonal to another element $y \in \mathbb{H},$ written as $x\bot_{A}^{\epsilon}y,$ if $|\langle x, y\rangle_A| \leq \epsilon \|x\|_A \|y\|_A.$
\end{definition}
It is easy to see that if $x \in N(A),$ then $x\bot_{A}y$ for all $y \in \mathbb{H}$ and therefore $x\bot_{A}^{\epsilon}y,$ for all $y \in \mathbb{H}$ and for any $\epsilon \in [0,1).$ Also, if we consider $A=I,$ the above definition coincides with the usual definition of approximate $\epsilon-$orthogonality in a Hilbert space.
\begin{definition}
  Let $\mathbb{H}$ be a Hilbert space and  $\epsilon \in [0,1).$ An element $x \in \mathbb{H}$ is said to be $(\epsilon,A)$-approximate orthogonal in the sense of Chmieli$\acute{n}$ski to another element $y \in \mathbb{H},$  written as $x\bot_{\epsilon(A)}y,$ if $\|x+\lambda y\|_A^2 \geq \|x\|_A^2 -2 \epsilon \|x\|_A \|\lambda y\|_A$ for all $\lambda \in \mathbb{K}.$
\end{definition}
Clearly, if we consider $A=I$ and $\epsilon = 0$, the above definition is equivalent to the usual definition of Birkhoff-James orthogonality in a Hilbert space. Following similar motivation, we introduce  $(\epsilon,A)$-approximate orthogonality in the sense of Chmieli$\acute{n}$ski, for $A$-bounded operators in the following way:
\begin{definition}
Let  $T \in B_{A^{1/2}}(\mathbb{H})$  and $\epsilon \in [0,1).$  Then $T$ is said to be \textit{$(\epsilon,A)$-approximate orthogonal in the sense of Chmieli$\acute{n}$ski} to another element $S \in B_{A^{1/2}}(\mathbb{H}),$  written as $T\bot_{\epsilon(A)}S,$ if $\|T+\lambda S\|_A^2 \geq \|T\|_A^2 -2 \epsilon \|T\|_A \|\lambda S\|_A$ for all $\lambda \in \mathbb{K}.$
 \end{definition}
  In \cite{S}, the author introduced the notions of left symmetric and right symmetric points in $\mathbb{L(\mathbb{H})}$. Motivated by that we introduce the following  notions of symmetry for $A$- bounded operators:
     \begin{definition}\label{A approx left symmetry right symmetry def}
        Let $\mathbb{H}$ be a Hilbert space and let $\epsilon \in [0,1).$
        An element $T \in B_{A^{1/2}}(\mathbb{H})$ is said to be $(\epsilon,A)$-approximate right symmetric in the sense of Chmieli$\acute{n}$ski, if for any $S \in B_{A^{1/2}}(\mathbb{H}),$  $S\bot_{\epsilon(A)}T \Rightarrow T\bot_{\epsilon(A)}S.$ \\Similarly,
        an element $T \in B_{A^{1/2}}(\mathbb{H})$ is said to be $(\epsilon,A)$-approximate left symmetric in the sense of Chmieli$\acute{n}$ski, if for any $S \in B_{A^{1/2}}(\mathbb{H}),$  $T\bot_{\epsilon(A)}S \Rightarrow S\bot_{\epsilon(A)}T.$
  \end{definition}
Note that by putting $A =I,$ we obtain the following definitions of approximate right symmetric points and approximate left symmetric points in $ \mathbb{L(\mathbb{H})}: $  An element $T \in \mathbb{L(\mathbb{H})}$ is said to be \textit{approximate right symmetric point}, if for any $S \in \mathbb{L(\mathbb{H})},$  $S\bot^{\epsilon}_{B}T \Rightarrow T\bot^{\epsilon}_{B}S,$ where $\epsilon \in [0,1).$ Similarly, an element $T \in \mathbb{L(\mathbb{H})}$ is said to be \textit{approximate left symmetric point}, if for any $S \in \mathbb{L(\mathbb{H})},$  $T\bot^{\epsilon}_{B}S \Rightarrow S\bot^{\epsilon}_{B}T,$ $\epsilon \in [0,1).$ Taking $ \epsilon=0, $ we obtain the corresponding definitions of the  symmetric points in $ \mathbb{L(\mathbb{H})}. $
 We refer the reader to \cite{dskpam, smp} for further details on approximate Birkhoff-James orthogonality. In this article, our main goal is to explore the concept of approximate Birkhoff-James orthogonality of operators in the semi-Hilbertian structure.  In the following section we characterize $(\epsilon,A)$-approximate orthogonality of $A-$bounded operators in the semi-Hilbertian structure. In the final section we study $(\epsilon,A)$-approximate left and right symmetric operators in the setting of finite-dimensional as well as infinite dimensional Hilbert spaces.

\section{$(\epsilon,A)$-approximate orthogonality of operators}
In  \cite {psmm}, the authors introduced the notions of the positive part and the negative part of an element in a complex Banach space, along a
particular direction. In the same spirit, for the purpose of our investigation of $ (\epsilon,A)- $approximate orthogonality in the sense of Chmieli$\acute{n}$ski, we introduce the following:\\
 Let $x\in \mathbb{H}$ and let  $\alpha \in \mathbb{U}=\{ \gamma \in \mathbb{C}: |\gamma|=1$
,$~arg ~\gamma \in [0,\pi)\}.$ Then
$$(x)_\alpha^{A+}=\left\lbrace y\in \mathbb{H}:\|x+\lambda y\|_A \geq \|x\|_A ~\text{for all}~\lambda = t\alpha,t\geq 0\right\rbrace, $$
$$(x)_\alpha^{A-}=\left\lbrace y\in \mathbb{H}:\|x+\lambda y\|_A \geq \|x\|_A ~\text{for all}~\lambda = t\alpha,t\leq 0\right\rbrace. $$
 Using these definitions, we note the following results in the form of a proposition.
\begin{prop}\label{x+,x-}
 Let $\mathbb{H}$ be a  complex Hilbert space and let $x,y \in \mathbb{H}.$ Let $\alpha \in
 \mathbb{U}=\{ \gamma \in \mathbb{C}: |\gamma|=1 ,~arg ~\gamma \in [0,\pi)\}.$ Then for each $\alpha \in \mathbb{U},$ the following are true.
 \begin{itemize}
 \item [($i$)] either $y \in (x)_\alpha^{A+}$ or $y \in (x)_\alpha^{A-}$.

\item [($ii$)] $x \bot_Ay $ if and only if  $y \in (x)_\alpha^{A+}$ and $y \in (x)_\alpha^{A-}$ for each $\alpha \in \mathbb{U} $.

\item [($iii$)]  $y \in (x)_\alpha^{A+}$ implies $\eta y \in (\mu x)_\alpha^{A+}$ for all $\eta , \mu \geq 0$.

\item [($iv$)] $y \in (x)_\alpha^{A+}$ implies $-y \in (x)_\alpha^{A-}$ and $y \in (-x)_\alpha^{A-}$.

\item [($v$)] $y \in (x)_\alpha^{A-}$ implies $\eta y \in (\mu x)_\alpha^{A-}$ for all $\eta , \mu \geq 0$.

\item [($vi$)] $y \in (x)_\alpha^{A-}$ implies $-y \in (x)_\alpha^{A+}$ and $y \in (-x)_\alpha^{A+}$.
\end{itemize}
\end{prop}
\begin{proof}
We only prove the first statement, as all other statements follow trivially. Let $\alpha \in \mathbb{U}$ and let $y \notin (x)_\alpha^{A+}$.
We claim that $y \in (x)_\alpha^{A-}$. Suppose on the contrary that $y \notin (x)_\alpha^{A-}$. Then, there exists $\lambda_0 < 0 $ such that
$\|x+\lambda_0\alpha y\|_A < \|x\|_A$.
 By assumption,  $y \notin (x)_\alpha^{A+}$. So, there exists $\lambda > 0 $ such that $\|x+\lambda \alpha y\|_A < \|x\|_A$. As $\lambda_0 < 0
$ and $\lambda > 0 $, we can find $t \in (0,1)$ such that $0 = (1-t)\lambda_0 + t \lambda$. Thus, $x = (1-t)(x+\lambda_0 \alpha y) + t(x+ \lambda \alpha
y).$ Therefore,
$$ \|x\|_A \leq (1-t)\|x+\lambda_0 \alpha y\|_A + t \|x+\lambda \alpha y\|_A
  <  (1-t)\|x\|_A + t \|x\|_A
  = \|x\|_A,$$
   a contradiction.
    Hence, $y \in (x)_\alpha^{A+}$ or $y \in (x)_\alpha^{A-}$. This completes the proof of the proposition.
\end{proof}
\begin{remark}
Note that by taking $\alpha = 1$ in the above proposition, we obtain the analogous statements for  real Hilbert spaces.
\end{remark}
It is easy to see that approximate $\epsilon-$orthogonality in a Hilbert space is symmetric. In our next proposition we note that $\bot_{A}^{\epsilon}$ relation is also symmetric.
\begin{prop}\label{commutative}
    Let $\mathbb{H}$ be a Hilbert space. Let $\epsilon \in [0,1)$ and $x,y \in \mathbb{H}.$ Then  $x\bot_{A}^{\epsilon}y$ if and only if $y\bot_{A}^{\epsilon}x.$
\end{prop}
\begin{proof}
  The proof follows from the fact that $|\langle x, y\rangle_A| = |\langle y, x\rangle_A|$ for all $x,y \in \mathbb{H}.$
\end{proof}
It is easy to see that $A$-orthogonality is homogeneous, i.e., if $x,y \in \mathbb{H}$ are such that $x\bot_{A}y,$ then $\alpha x \bot_{A}\beta y$ for all $\alpha, \beta \in \mathbb{K}.$ Moreover, in \cite{C}, the author proved that approximate Birkhoff-James orthogonality in a Banach space (and therefore in  Hilbert space) is homogeneous. In fact, one can easily show that $\bot_{A}^{\epsilon}$ relation is also homogeneous. In our next proposition we observe that  $\bot_{\epsilon(A)}$ relation is homogeneous.
\begin{prop}\label{homogeneous}
   Let $\mathbb{H}$ be a Hilbert space and $\epsilon \in [0,1).$ Then $\bot_{\epsilon(A)}$ relation is homogeneous.
   \end{prop}
\begin{proof}
   Let $x,y \in \mathbb{H}$ be such that $x\bot_{\epsilon(A)}y.$ Let $\alpha , \beta \in \mathbb{K}$ (excluding the obvious case $\alpha = 0$). Then,
   $$\|\alpha x + \lambda \beta y\|_A^2
           \geq  |\alpha|^2 (\|x\|_A^2 - 2 \epsilon \|x\|_A \|\lambda \frac{\beta}{\alpha} y\|_A )
           = \|\alpha x\|_A^2 - 2 \epsilon \|\alpha x\|_A \|\lambda \beta y\|_A.$$
  This completes the proof of our proposition.
\end{proof}
Following  Proposition \ref{homogeneous} it is easy to see that  if $T,S \in B_{A^{1/2}}(\mathbb{H})$ be such that $T \bot_{\epsilon(A)} S,$ then also $\alpha T \bot_{\epsilon(A)} \beta S$ for all $\alpha, \beta \in \mathbb{K}.$
In \cite{CSW}, the authors proved that in a Hilbert space $\mathbb{H},$ for $x,y \in \mathbb{H},$  $x \bot^{\epsilon} y$ if and only if there exists $z \in \mathbb{H}$ such that $x \bot_B z$ and $\|z-y\| \leq \epsilon \|y\|.$
 In the following theorem, we characterize $(\epsilon,A)-$approximate orthogonality  in the semi-Hilbertian structure.
\begin{theorem}\label{characterization}
  Let $\mathbb{H}$ be a Hilbert space. Let $\epsilon \in [0,1)$ and $x,y \in \mathbb{H}.$ Then the following conditions are equivalent:
  \begin{itemize}
    \item [($i$)]$ x\bot_{A}^{\epsilon}y$
    \item [($ii$)] there exists an element $z \in \mathbb{H}$ such that $x \bot_{A} z$ and $\|z-y\|_A \leq \epsilon \|y\|_A.$
  \end{itemize}
\end{theorem}
\begin{proof}
  $($i$) \Rightarrow ($ii$): $ If $\|x\|_A=0,$ i.e., $x \in N(A),$ choose $z=y.$ Clearly, $x \bot_{A} z$ and $0 = \|z-y\|_A \leq \epsilon \|y\|_A.$\\
  Next, we assume that $\|x\|_A \neq 0.$ Choose $z = \frac{-\overline{\langle x,y \rangle_A}}{\|x\|_A^2}x +y.$
  Clearly, $\langle x,z \rangle_A = -\frac{\langle x,y \rangle_A}{\|x\|_A^2}\langle x,x\rangle_A + \langle x,y \rangle_A = 0.$
  Therefore, $x \bot_{A} z.$
   Again,
   $$\|z-y\|_A = \frac{|-\overline{\langle x,y \rangle_A}|}{\|x\|_A^2}\|x\|_A
     \leq  \frac{\epsilon \|x\|_A^2 \|y\|_A}{\|x\|_A^2}
     =\epsilon \|y\|_A.$$

   $($ii$) \Rightarrow ($i$): $ Let $x \bot_{A} z$ be such that $\|z-y\|_A \leq \epsilon \|y\|_A.$
   Therefore, $$ |\langle x,y\rangle_A| =  |\langle x,y-z\rangle_A|
      \leq \|x\|_A \|y-z\|_A
      \leq \epsilon  \|x\|_A \|y\|_A.$$

   This completes the proof of the theorem.
\end{proof}
 Clearly, in Theorem \ref{characterization}, if we consider $A =I,$ we obtain the characterization of approximate $\epsilon-$orthogonality \cite{CSW} in a Hilbert space. Thus, Theorem \ref{characterization} generalizes the concept of approximate $\epsilon-$orthogonality in a Hilbert space.\\
 It is easy to see that approximate $\epsilon-$orthogonality and approximate Birkhoff-James orthogonality are same in a Hilbert space. Our next aim is to show the equivalence of $(\epsilon,A)-$approximate orthogonality and $(\epsilon,A)-$approximate orthogonality in the sense of Chmieli$\acute{n}$ski in Hilbert space setting. To do so we need the following easy proposition.
\begin{prop}\label{re inner product a x,y}
   Let $\mathbb{H}$ be a Hilbert space. Let $x \in S_{\mathbb{H}(A)}$ and $y \in \mathbb{H}.$ Then $$\lim_{\lambda \in \mathbb{R}, \lambda \to 0} \frac{\|x+\lambda y\|_A -1}{\lambda}= Re~\langle x,y\rangle_A.$$
\end{prop}
\begin{proof}
  Clearly,
\begin{eqnarray*}
     \lim_{\lambda \in \mathbb{R}, \lambda \to 0} \frac{\|x+\lambda y\|_A -1}{\lambda} &=& \lim_{\lambda \in \mathbb{R}, \lambda \to 0} \frac{\|x+\lambda y\|^2_A -1}{\lambda(\|x+\lambda y\|_A +1)}\\
     &=& \lim_{\lambda \in \mathbb{R}, \lambda \to 0} \frac{\lambda(2 Re \langle x,y\rangle_A + \lambda \|y\|_A^2)}{\lambda (\|x+\lambda y\|_A +1)}\\
     &=& Re\langle x,y\rangle_A.
  \end{eqnarray*}
\end{proof}
In our next theorem we establish the equivalence of $\bot_{A}^{\epsilon}$ and $\bot_{\epsilon(A)}$ in the semi-Hilbertian structure, which generalizes \cite[Prop.2.1]{C} in the setting of Hilbert space.
\begin{theorem}\label{equality}
 Let $\mathbb{H}$ be a Hilbert space and $\epsilon \in [0,1).$ Let $x,y \in \mathbb{H}.$ Then $x\bot_{A}^{\epsilon}y \Leftrightarrow  x\bot_{\epsilon(A)}y.$
\end{theorem}
\begin{proof}
  Let $x,y \in \mathbb{H}$ be such that $x\bot_{A}^{\epsilon}y.$ Therefore,
  \begin{eqnarray*}
    \|x+\lambda y\|_A^2  &=& \|x\|_A^2 + 2 Re (\overline{\lambda} \langle x,y\rangle_A) + |\lambda|^2 \|y\|_A^2\\
     &\geq& \|x\|_A^2 - 2  |\overline{\lambda} \langle x,y\rangle_A|  \\
   &\geq& \|x\|_A^2 -2 \epsilon \|x\|_A \|\lambda y\|_A.
  \end{eqnarray*}
  Hence $x\bot_{\epsilon(A)}y.$ Thus,  $\bot_{A}^{\epsilon} \subseteq \bot_{\epsilon(A)}$.\\
  Conversely, let $x,y \in \mathbb{H}$ be such that $x\bot_{\epsilon(A)}y.$ If either $\|x\|_A$ or $\|y\|_A =0$, the result holds trivially. So, we assume $\|x\|_A,\|y\|_A \neq 0.$ Since $\bot_{\epsilon(A)}$ relation is homogeneous, without loss of generality we may assume that $\|x\|_A=\|y\|_A=1.$ Therefore, $\|x+\lambda y\|_A^2-1 +2 \epsilon |\lambda| \geq 0$ for all $\lambda \in \mathbb{K}.$ Hence, $\langle x,x+\lambda y\rangle_A + \langle \lambda y,x+\lambda y\rangle_A$ is real and so we have $Re\langle x,x+\lambda y\rangle_A + Re\langle \lambda y,x+\lambda y\rangle_A -1 +2 \epsilon |\lambda| \geq 0 $ for all $\lambda \in \mathbb{K}.$ It is easy to see that $Re\langle x,x+\lambda y\rangle_A \leq |\langle x,x+\lambda y\rangle_A| \leq \|x+\lambda y\|_A.$ Thus, we have, $$\|x+\lambda y\|_A +Re \langle \lambda y, x+\lambda y\rangle_A -1 + 2\epsilon |\lambda| \geq 0.$$ Let $\lambda_0 \in \mathbb{K}\setminus \{0\}$ and $n \in \mathbb{N}$ and $\lambda_n = \frac{\lambda_0}{n}.$ Clearly, $$\|x+\lambda_n y\|_A +Re \langle \lambda_n y, x+\lambda_n y\rangle_A -1 + 2\epsilon |\lambda_n| \geq 0.$$ Hence, $$\|x+ \frac{\lambda_0}{n} y\|_A +Re \langle \frac{\lambda_0}{n} y, x+\frac{\lambda_0}{n} y\rangle_A -1 \geq - 2\epsilon |\frac{\lambda_0}{n}|.$$ So, $$Re \langle \frac{\lambda_0}{|\lambda_0|} y, x+\frac{\lambda_0}{n} y\rangle_A + \frac{\|x+ \frac{\lambda_0}{n} y\|_A -1}{\frac{|\lambda_0|}{n}} \geq - 2\epsilon .$$ Let $\frac{\lambda_0}{|\lambda_0|} y = y^{'}$ and $\frac{|\lambda_0|}{n}=\xi_{n}.$ It is easy to see that $\xi_{n} \longrightarrow 0$ as $n \longrightarrow \infty.$ Thus, $$Re \langle y^{'} , x+ \xi_{n} y^{'}\rangle_A + \frac{\|x+ \xi_{n} y^{'} \|_A -1}{\xi_{n}} \geq - 2\epsilon .$$ Therefore, by taking limit and applying Proposition \ref{re inner product a x,y}, we get $2Re \langle y^{'} , x\rangle_A \geq - 2\epsilon,$ i.e., $Re \langle \lambda_{0}y , x\rangle_A \geq - \epsilon |\lambda_{0}|.$ Similarly by using $-\lambda_{0}$ instead of $\lambda_{0}$, we get $Re \langle \lambda_{0}y , x\rangle_A \leq  \epsilon |\lambda_{0}|.$ Therefore, $|Re \langle \lambda_{0}y , x\rangle_A| \leq  \epsilon |\lambda_{0}|.$ Choosing $\lambda_{0}=\overline{\langle y , x\rangle_A},$ we obtain $ |\langle y , x\rangle_A| \leq  \epsilon.$ Hence, by Proposition \ref{commutative}, we have $x~\bot_{A}^{\epsilon}~y.$ This completes the proof of the theorem.
\end{proof}
Next we characterize the $(\epsilon,A)-$approximate orthogonality in the sense of Chmieli$\acute{n}$ski for $A$-bounded operators in a complex Hilbert space.
\begin{theorem}\label{bounded}
 Let $\mathbb{H}$ be a complex Hilbert space  and  $\epsilon \in [0,1).$  Let $T,S \in B_{A^{1/2}}(\mathbb{H}).$  Then the following conditions are equivalent:
  \begin{itemize}
    \item [(i)] $T \bot_{\epsilon(A)} S$
    \item [(ii)] for each $\theta \in [0,\pi),$ there exist sequences $\{x_{n(\theta)}\}, \{y_{n(\theta)}\}\subseteq S_{\mathbb{H}(A)}$ such that
    \begin{itemize}
      \item [(a)] $\lim_{n \to \infty} \|Tx_{n(\theta)}\|_A = \lim_{n \to \infty} \|Ty_{n(\theta)}\|_A = \|T\|_A.$
      \item [(b)] $\lim_{n \to \infty}Re(e^{-i\theta} \langle Tx_{n(\theta)},Sx_{n(\theta)}\rangle_A) \geq -\epsilon \|T\|_A \|S\|_A.$
      \item [(c)] $\lim_{n \to \infty}Re(e^{-i\theta} \langle Ty_{n(\theta)},Sy_{n(\theta)}\rangle_A) \leq \epsilon \|T\|_A \|S\|_A.$
    \end{itemize}
  \end{itemize}
\end{theorem}
\begin{proof}
  Suppose (i) holds.
  Therefore, $\|T+\lambda S\|_A^2 \geq \|T\|_A^2 -2 \epsilon \|T\|_A \|\lambda S\|_A$ for all $\lambda \in \mathbb{K}.$
   First we choose a sequence $\{\lambda_n\} \subseteq \mathbb{K}$, where $\lambda_n=\frac{e^{i\theta}}{n},$ where $\theta \in [0,\pi)$ and $n \in \mathbb{N}.$
   Clearly, for each $n \in \mathbb{N},$ there exists an element $x_{n(\theta)} \in S_{\mathbb{H}(A)}$ such that
   \begin{eqnarray*}
     \|(T+\frac{e^{i\theta}}{n}S)x_{n(\theta)}\|_A^2 &>& \|T\|_A^2 - 2 \frac{\epsilon}{n}\|T\|_A \|S\|_A - \frac{1}{n^2}.
   \end{eqnarray*}
    Hence,
    $$\|Tx_{n(\theta)}\|_A^2 + \frac{2}{n}Re(e^{-i\theta}\langle Tx_{n(\theta)},Sx_{n(\theta)}\rangle_A) + \frac{1}{n^2}\|Sx_{n(\theta)}\|_A^2 > \|T\|_A^2 - 2 \frac{\epsilon}{n}\|T\|_A \|S\|_A - \frac{1}{n^2}.$$
   Thus,
    $$\|Tx_{n(\theta)}\|_A^2 > \|T\|_A^2 -\frac{2}{n}Re(e^{-i\theta}\langle Tx_{n(\theta)},Sx_{n(\theta)}\rangle_A) - \frac{1}{n^2}\|S\|_A^2 - 2 \frac{\epsilon}{n}\|T\|_A \|S\|_A - \frac{1}{n^2}.$$

      It is easy to see that $\{\|Tx_{n(\theta)}\|_A\},\{\langle Tx_{n(\theta)},Sx_{n(\theta)}\rangle_A\}$ are  bounded sequences.
       So, passing onto a subsequence if necessary, it follows that $\lim_{n \to \infty} \|Tx_{n(\theta)}\|_A^2 \\\geq \|T\|_A^2.$
       Again by the fact $x_{n(\theta)} \in S_{\mathbb{H}(A)},$ clearly,  $\|Tx_{n(\theta)}\|_A^2 \leq \|T\|_A^2$ for all $n \in \mathbb{N}.$
       Therefore, $\lim_{n \to \infty} \|Tx_{n(\theta)}\|_A = \|T\|_A.$
       Thus,  $$ \frac{2}{n}Re(e^{-i\theta}\langle Tx_{n(\theta)},Sx_{n(\theta)}\rangle_A)> (\|T\|_A^2 -\|Tx_{n(\theta)}\|_A^2) - 2 \frac{\epsilon}{n}\|T\|_A \|S\|_A -  \frac{1}{n^2}\|S\|_A^2 - \frac{1}{n^2}.$$
       So, $$ Re(e^{-i\theta}\langle Tx_{n(\theta)},Sx_{n(\theta)}\rangle_A)> \frac{n}{2}(\|T\|_A^2 -\|Tx_{n(\theta)}\|_A^2) -  \epsilon\|T\|_A \|S\|_A -  \frac{1}{2n}\|S\|_A^2 - \frac{1}{2n}.$$
       Therefore,$$ Re(e^{-i\theta}\langle Tx_{n(\theta)},Sx_{n(\theta)}\rangle_A)> -  \epsilon\|T\|_A \|S\|_A -  \frac{1}{2n}\|S\|_A^2 - \frac{1}{2n}.$$ Hence, by passing onto a subsequence if necessary and by taking limit, we obtain
      \begin{eqnarray*}
       \lim_{n \to \infty} Re(e^{-i\theta}\langle Tx_{n(\theta)},Sx_{n(\theta)}\rangle_A) &\geq& - \epsilon\|T\|_A \|S\|_A.
      \end{eqnarray*}
      Similarly, for $\lambda_n=\frac{-e^{i\theta}}{n},$ we can find a sequence $\{y_{n(\theta)}\} \subseteq S_{\mathbb{H}(A)},$ such that $$\lim_{n \to \infty} \|Ty_{n(\theta)}\|_A = \|T\|_A$$ and
      \begin{eqnarray*}
       \lim_{n \to \infty} Re(e^{-i\theta}\langle Ty_{n(\theta)},Sy_{n(\theta)}\rangle_A) &\leq&  \epsilon\|T\|_A \|S\|_A.
      \end{eqnarray*}
      This completes the proof of  (i) $\Rightarrow $ (ii).

     Next suppose (ii) holds. First, we may assume that $\lambda = |\lambda|e^{i\theta} \in \mathbb{K}$ where $\theta \in [0,\pi).$ Therefore,
      \begin{eqnarray*}
        \|T+\lambda S\|_A^2 &\geq& \lim_{n \to \infty}\|(T+|\lambda|e^{i\theta}S)x_{n(\theta)}\|_A^2 \\
         &=& \lim_{n \to \infty} (\|Tx_{n(\theta)}\|_A^2 + 2 |\lambda| Re(e^{-i\theta}\langle Tx_{n(\theta)},Sx_{n(\theta)}\rangle_A) + |\lambda|^2 \|Sx_{n(\theta)}\|_A^2)\\
         &\geq&\lim_{n \to \infty}(\|Tx_{n(\theta)}\|_A^2 + 2 |\lambda| Re(e^{-i\theta}\langle Tx_{n(\theta)},Sx_{n(\theta)}\rangle_A) \\
          &\geq& \|T\|_A^2 - 2 \epsilon\|T\|_A \|\lambda S\|_A.
        \end{eqnarray*}
        Similarly, for  $\lambda = -|\lambda|e^{i\theta} \in \mathbb{K}$ where $\theta \in [0,\pi),$ we can show that $\|T+\lambda S\|_A^2 \geq \|T\|_A^2 - 2 \epsilon\|T\|_A \|\lambda S\|_A.$
      Thus $T \bot_{\epsilon(A)} S$  and so (ii) $ \Rightarrow $ (i). This completes the proof of the theorem.
\end{proof}
The characterization of $(\epsilon,A)$-approximate orthogonality in the sense of Chmieli$\acute{n}$ski for $A-$bounded operators in a real Hilbert space admits a nicer form due to the fact that in real normed space $\mathcal{N}$ and for $x,y \in \mathcal{N},$ $x \bot^{\epsilon}_{B} y$ if and only if there exists $z \in \langle\{x,y\}\rangle$ such that $x \bot_{B} z$ and $\|z-y\| \leq \epsilon \|y\|$ (Theorem 2.2 of \cite{CSW}). In order to characterize $(\epsilon,A)$-approximate orthogonality in the sense of Chmieli$\acute{n}$ski for $A-$bounded operators in a real Hilbert space first we prove the following proposition. 
\begin{prop}\label{helpful}
	Let $\mathbb{H}$ be a real Hilbert space. Let $P$ be the orthogonal projection on $\overline{R(A)}.$ Let $A_0 = A\mid_{\overline{R(A)}}.$ Let $C \in  B_{A^{1/2}}(\mathbb{H})$ and $\tilde{C}= PC\mid_{\overline{R(A)}}.$  Then,
	\begin{itemize}
		\item [($i$)] $\|C\|_A = \|\tilde{C}\|_{A_0}$ for all $C \in  B_{A^{1/2}}(\mathbb{H}).$
		\item [($ii$)] \~{T} +\~{S} =$\widetilde{T+S}$ for all $T,S \in  B_{A^{1/2}}(\mathbb{H}).$
		\item [($iii$)] $\widetilde{\lambda C} = \lambda \tilde{C}$ for all $C \in  B_{A^{1/2}}(\mathbb{H})$ and $\lambda \in \mathbb{R}.$
		\item [($iv$)] $M_{A}^C \cap \overline{R(A)} = M_{A_0}^{\tilde{C}}$ for all $C \in  B_{A^{1/2}}(\mathbb{H}).$
		\item [($v$)] $S_{\mathbb{H}(A)} \cap \overline{R(A)} = S_{\overline{R(A)}(A_0)}.$
		\item [($vi$)] If $M_{A}^C \cap \overline{R(A)} =S_{\mathbb{H}(A)} \cap \overline{R(A)}$ for some $C \in  B_{A^{1/2}}(\mathbb{H}),$ then $M_{A}^C=S_{\mathbb{H}(A)}.$
	\end{itemize}
\end{prop}
\begin{proof}
	Note that, if $z \in N(A)$ and $C \in  B_{A^{1/2}}(\mathbb{H}),$ clearly, $\|Cz\|_A = 0.$\\
	($i$) Let $C \in  B_{A^{1/2}}(\mathbb{H}).$ Clearly,
	\begin{eqnarray*}
		\|C\|_A &=& \sup\{\|Cx\|_A : \|x\|_A =1\}\\
		&=& \sup\{\|Cu+Cv\|_A : \|u+v\|_A =1, x = u+v,u \in N(A),v\in \overline{R(A)}\}\\
		&=& \sup\{\|Cv\|_A : \|v\|_A =1, v\in \overline{R(A)}\}\\
		&=& \sup\{\|PCv\|_A : \|v\|_A =1, v\in \overline{R(A)}\}\\
		&=& \sup\{\|\tilde{C}v\|_{A_0} : \|v\|_{A_0} =1, v\in \overline{R(A)}\}\\
		&=& \|\tilde{C}\|_{A_0}.
	\end{eqnarray*}
	We omit the proof of ($ii$),($iii$),($iv$),($v$); as they follow trivially.\\
	($vi$) Suppose on the contrary that  $M_{A}^C \neq S_{\mathbb{H}(A)}.$ Then, there exists $x \in S_{\mathbb{H}(A)}$ such that $x \notin M_{A}^C.$ As $\mathbb{H} = N(A) \oplus \overline{R(A)},$ $x$ can be uniquely written as $x = u+v,$ where $u \in N(A)$ and $v \in \overline{R(A)}.$ As $u \in N(A),$  $\|u\|_A =0$ and therefore, $\|x\|_A = \|v\|_A =1.$ Hence, $v \in S_{\mathbb{H}(A)} \cap \overline{R(A)} =M_{A}^C \cap \overline{R(A)}.$ So, $\|Cv\|_A =\|C\|_A.$ As $\|Cu\|_A =0, $ it follows that $\|Cv\|_A = \|Cx\|_A =\|C\|_A.$ Thus, $x \in M_{A}^C,$ a contradiction.
\end{proof}
\begin{remark}\label{t implies t tilda}
	It is clear from ($i$) of Proposition \ref{helpful} that for any  $T,S \in  B_{A^{1/2}}(\mathbb{H})$ and $\epsilon \in [0,1),$ $S \bot_{\epsilon(A)} T \Leftrightarrow \tilde{S} \bot_{\epsilon(A_0)} \tilde{T}$ and $S \bot_A^B T \Leftrightarrow \tilde{S} \bot_{A_0}^B \tilde{T}.$
\end{remark}
Now we are ready to characterize $(\epsilon,A)$-approximate orthogonality in the sense of Chmieli$\acute{n}$ski for $A-$bounded operators in  real Hilbert space.
\begin{theorem}\label{real A bounded}
	Let $\mathbb{H}$ be a real Hilbert space and $T,S \in B_{A^{1/2}}(\mathbb{H}).$ Let $\epsilon \in [0,1).$ Then the following conditions are equivalent:
 \begin{itemize}
	\item [($i$)] $T ~\bot_{\epsilon(A)}~ S$
	\item [($ii$)]  there exists a sequence $\{x_n\} \subseteq S_{\mathbb{H}(A)}$ such that
	\begin{itemize}
		\item [($a$)] $\lim_{n \to \infty} \|Tx_n\|_A = \|T\|_A.$
		\item [($b$)] $\lim_{n \to \infty} |\langle Tx_n,Sx_n \rangle_A| \leq \epsilon \|T\|_A \|S\|_A.$
	\end{itemize}
\end{itemize}
\end{theorem}
\begin{proof}
	As the sufficient part of the theorem follows trivially, we only prove the necessary part of the theorem.\\
As $A$ is positive definite on $\overline{R(A)},$ $\big(\overline{R(A)},\|\cdot\|_{A}\big)$ is a normed space. Let $A_0 = A\mid_{\overline{R(A)}}.$ Let $P$ be the orthogonal projection on $\overline{R(A)}.$ Let $\tilde{T}= PT\mid_{\overline{R(A)}}.$ Clearly, $\big(B_{A_{0}^{1/2}}(\overline{R(A)}),\|\cdot\|_{A_0}\big)$ is a normed space. As $T ~\bot_{\epsilon(A)}~ S,$ it is easy to see that $\tilde{T} ~\bot_{\epsilon(A_0)}~ \tilde{S}.$ Therefore, by \cite[Th.2.2]{CSW} , there exists $\tilde{U} \in B_{A_{0}^{1/2}}(\overline{R(A)})$ such that $\tilde{T} \bot_{A_0}^B \tilde{U}$ and $\|\tilde{S} - \tilde{U}\|_{A_0} \leq \epsilon \|\tilde{S}\|_{A_0}.$ As $\tilde{T} \bot_{A_0}^B \tilde{U}$, by \cite[Th.2.2]{z}, there exists a sequence of $A_0-$unit vectors $\{x_n\}$ in $\overline{R(A)}$ such that
$\lim_{n \to \infty}\|\tilde{T}x_n\|_{A_0} =\|\tilde{T}\|_{A_0}$ and $\lim_{n \to \infty}\langle \tilde{T}x_n,\tilde{U}x_n \rangle_{A_0} = 0.$
Thus
\begin{eqnarray*}
	|\langle \tilde{T}x_n,\tilde{S}x_n \rangle_{A_0}| &\leq& |\langle \tilde{T}x_n,\tilde{S}x_n - \tilde{U}x_n \rangle_{A_0}| + |\langle \tilde{T}x_n,\tilde{U}x_n \rangle_{A_0}|\\
	&\leq& \|\tilde{T}\|_{A_0}\|\tilde{S} - \tilde{U}\|_{A_0} +  |\langle \tilde{T}x_n,\tilde{U}x_n \rangle_{A_0}|
\end{eqnarray*}
Hence, by taking limit, we obtain $\lim_{n \to \infty} |\langle \tilde{T}x_n,\tilde{S}x_n \rangle_{A_0}| \leq \epsilon \|\tilde{T}\|_{A_0} \|\tilde{S}\|_{A_0}.$ As $A$ is positive, $N(A)=N(A^{1/2})$. Thus, we have, $ \langle \tilde{T}x_n,\tilde{S}x_n \rangle_{A_0}= \langle Tx_n,Sx_n \rangle_A$ and $\|Tx_n\|_A = \|\tilde{T}x_n\|_{A_0}.$ Also note that $\|T\|_A = \|\tilde{T}\|_{A_0}.$ Therefore, $\lim_{n \to \infty} |\langle Tx_n,Sx_n \rangle_A|\\
 \leq \epsilon \|T\|_A \|S\|_A.$
\end{proof}
\begin{remark}
In \cite[Th.3.2]{CSW}, the authors proved that if $T,S \in \mathbb{L(\mathbb{H})},$ where $\mathbb{H}$ is a real Hilbert space and $\epsilon \in [0,1),$ then $T \bot_{\epsilon}^B S$
if and only if there exists a sequence $\{x_n\} \subseteq S_{\mathbb{H}}$ such that
$\lim_{n \to \infty} \|Tx_n\| = \|T\|$ and
	 $\lim_{n \to \infty} |\langle Tx_n,Sx_n \rangle| \leq \epsilon \|T\| \|S\|.$ Thus Theorem \ref{real A bounded} generalizes  \cite[Th.3.2]{CSW}.
	
\end{remark}
Next we characterize the $(\epsilon,A)-$approximate orthogonality in the sense of Chmieli$\acute{n}$ski for $A$-bounded compact operators in complex Hilbert space setting.
\begin{theorem}\label{compact}
  Let $\mathbb{H}$ be a complex Hilbert space  and $\epsilon \in [0,1).$   Suppose $B_{\mathbb{H}(A)} \cap \overline{R(A)}$ is bounded with respect to $\|\cdot\|$. Let $T,S \in B_{A^{1/2}}(\mathbb{H}) \cap \mathbb{K(\mathbb{H})}$.  Then the following conditions are equivalent:
  \begin{itemize}
    \item [(i)] $T \bot_{\epsilon(A)} S$
    \item [(ii)] for each $\theta \in [0,\pi),$ there exist $x_{\theta},y_{\theta} \in S_{\mathbb{H}(A)}$ such that
    \begin{itemize}
      \item [(a)] $ \|Tx_{\theta}\|_A = \|Ty_{\theta}\|_A =\|T\|_A.$
      \item [(b)] $Re(e^{-i\theta} \langle Tx_{\theta},Sx_{\theta}\rangle_A) \geq -\epsilon \|T\|_A \|S\|_A.$
        \item [(c)] $Re(e^{-i\theta} \langle Ty_{\theta},Sy_{\theta}\rangle_A) \leq \epsilon \|T\|_A \|S\|_A.$
      \end{itemize}
      \end{itemize}
      \end{theorem}
      \begin{proof}
        (ii) $\Rightarrow$ (i):  First we may assume that $\lambda = |\lambda|e^{i\theta} \in \mathbb{K},$ where $\theta \in [0,\pi).$ Therefore,
      \begin{eqnarray*}
        \|T+\lambda S\|_A^2 &\geq& \|(T+|\lambda|e^{i\theta}S)x_{\theta}\|_A^2 \\
         &=& \|Tx_{\theta}\|_A^2 + 2 |\lambda| Re(e^{-i\theta}\langle Tx_{\theta},Sx_{\theta}\rangle_A) + |\lambda|^2 \|Sx_{\theta}\|_A^2\\
        &\geq&  \|T\|_A^2 - 2 \epsilon\|T\|_A \|\lambda S\|_A.
          \end{eqnarray*}
         Similarly, for $\lambda = -|\lambda|e^{i\theta} \in \mathbb{K}$ where $\theta \in [0,\pi),$ we can show that $\|T+\lambda S\|_A^2 \geq \|T\|_A^2 - 2 \epsilon\|T\|_A \|\lambda S\|_A.$ Hence, $T \bot_{\epsilon(A)} S.$\\
         (i) $\Rightarrow$ (ii): By Theorem \ref{bounded}, for each $\theta \in [0,\pi),$ there exist  sequences $\{x_{n(\theta)}\},\{y_{n(\theta)}\} \subseteq S_{\mathbb{H}(A)}$ such that
    \begin{itemize}
      \item [(a)] $\lim_{n \to \infty} \|Tx_{n(\theta)}\|_A = \lim_{n \to \infty} \|Ty_{n(\theta)}\|_A = \|T\|_A.$
      \item [(b)] $\lim_{n \to \infty}Re(e^{-i\theta} \langle Tx_{n(\theta)},Sx_{n(\theta)}\rangle_A) \geq -\epsilon \|T\|_A \|S\|_A.$
       \item [(c)] $\lim_{n \to \infty}Re(e^{-i\theta} \langle Ty_{n(\theta)},Sy_{n(\theta)}\rangle_A) \leq \epsilon \|T\|_A \|S\|_A.$
      \end{itemize}
      As $A \in \mathbb{L(\mathbb{H})}$ is positive, $\mathbb{H}=N(A)\oplus \overline{R(A)}.$ Therefore, each $x_{n(\theta)}$ can be uniquely written as $x_{n(\theta)}=u_{n(\theta)}+ v_{n(\theta)},$ where $u_{n(\theta)} \in N(A)$ and $v_{n(\theta)} \in \overline{R(A)}$. As $u_{n(\theta)} \in N(A),$ it is easy to see that $\|u_{n(\theta)}\|_A=0$ and therefore, $\|x_{n(\theta)}\|_A=\|v_{n(\theta)}\|_A=1$ for each $n \in \mathbb{N}.$ Thus, we conclude that $\{v_{n(\theta)}\} \subseteq S_{\mathbb{H}(A)} \cap \overline{R(A)}.$  Since $T,S \in B_{A^{1/2}}(\mathbb{H})$, it is easy to see that $\|Tu_{n(\theta)}\|_A = \|Su_{n(\theta)}\|_A = 0$ and therefore $Tu_{n(\theta)},Su_{n(\theta)} \in N(A)$  for all $n \in \mathbb{N}$. Hence, $\|Tx_{n(\theta)}\|_A = \|Tv_{n(\theta)}\|_A$ and $\|Sx_{n(\theta)}\|_A = \|Sv_{n(\theta)}\|_A$ for each $n \in \mathbb{N}.$ As $A$ is positive, $N(A)=N(A^{1/2})$ and so $A^{1/2}(Tu_{n(\theta)}) = A^{1/2}(Su_{n(\theta)}) =\theta.$ Clearly, $\lim_{n \to \infty} \|Tx_{n(\theta)}\|_A = \lim_{n \to \infty} \|Tv_{n(\theta)}\|_A = \|T\|_A.$ Also,
      \begin{eqnarray*}
      -\epsilon\|T\|_A \|S\|_A &\leq& \lim_{n \to \infty} Re(e^{-i\theta}\langle Tx_{n(\theta)},Sx_{n(\theta)}\rangle_A)\\
     &=& \lim_{n \to \infty}Re (e^{-i\theta}\langle A^{1/2}(Tu_{n(\theta)} + Tv_{n(\theta)}),A^{1/2}(Su_{n(\theta)} + Sv_{n(\theta)}) \rangle)  \\
      &=& \lim_{n \to \infty}Re(e^{-i\theta}\langle  A^{1/2}Tv_{n(\theta)}, A^{1/2}Sv_{n(\theta)} \rangle)\\
     &=& \lim_{n \to \infty}Re(e^{-i\theta}\langle Tv_{n(\theta)},Sv_{n(\theta)}\rangle_A)\\
       \end{eqnarray*}
  Since, $\mathbb{H}$ is reflexive and $B_{\mathbb{H}(A)} \cap \overline{R(A)}$ is closed, convex and bounded with respect to $\|\cdot\|,$  $B_{\mathbb{H}(A)} \cap \overline{R(A)}$ is weakly compact with respect to $\|\cdot\|.$  Thus, the sequence $\{v_{n(\theta)}\}$ has a weakly convergent subsequence. Without loss of generality we may assume that $v_{n(\theta)} \rightharpoonup x_{\theta}$ with respect to $\|\cdot\|$ in $\mathbb{H}$, for some $x_{\theta} \in B_{\mathbb{H}(A)} \cap \overline{R(A)}.$ Since, $T,S \in \mathbb{K}(\mathbb{H}),$ it follows that $Tv_{n(\theta)} \longrightarrow Tx_{\theta}$ and $Sv_{n(\theta)} \longrightarrow Sx_{\theta}$ with respect to $\|\cdot\|$ in $\mathbb{H}$. Therefore,
  $$\lim_{n \to \infty} \|Tx_{n(\theta)}\|_A = \lim_{n \to \infty} \|Tv_{n(\theta)}\|_A = \|Tx_{\theta}\|_A=\|T\|_A.$$ Thus, $x_{\theta} \in S_{\mathbb{H}(A)}.$ Also we have,
  \begin{eqnarray*}
       \lim_{n \to \infty} Re(e^{-i\theta}\langle Tx_{n(\theta)},Sx_{n(\theta)}\rangle_A) &=& \lim_{n \to \infty} Re(e^{-i\theta}\langle Tv_{n(\theta)},Sv_{n(\theta)}\rangle_A)\\
      &=& Re(e^{-i\theta}\langle Tx_{\theta},Sx_{\theta}\rangle_A)\\
       &\geq&  -\epsilon\|T\|_A \|S\|_A.
      \end{eqnarray*}
      Similarly, we can find $y_{\theta} \in S_{\mathbb{H}(A)}$ such that $\|Ty_{\theta}\|_A=\|T\|_A$ and $Re(e^{-i\theta}\langle Ty_{\theta},Sy_{\theta}\rangle_A)
       \leq  \epsilon\|T\|_A \|S\|_A.$
      \end{proof}
      The following corollary is an easy consequence of the above theorem.
      \begin{corollary}\label{compact single point}
  Let $\mathbb{H}$ be a complex Hilbert space and  $\epsilon \in [0,1).$  Suppose $B_{\mathbb{H}(A)} \cap \overline{R(A)}$ is bounded with respect to $\|\cdot\|$. Let $T,S \in B_{A^{1/2}}(\mathbb{H}) \cap \mathbb{K(\mathbb{H})}$ and $M_A^T = \{\pm e^{i\alpha} x, ~\alpha \in [0,\pi)\}$.  Then $T \bot_{\epsilon(A)} S$ if and only if $ \|Tx\|_A  =\|T\|_A\ $ and \newline $ |Re(e^{-i\theta} \langle Tx,Sx\rangle_A)| \leq \epsilon \|T\|_A \|S\|_A$ for each $\theta \in [0,\pi)$.
\end{corollary}
 The characterization of $(\epsilon,A)$-approximate orthogonality in the sense of Chmieli$\acute{n}$ski for $A-$bounded compact operators in a complex Hilbert space is obtained by two $A-$unit vectors (Theorem \ref{compact}). But in a real Hilbert space, $(\epsilon,A)$-approximate orthogonality in the sense of Chmieli$\acute{n}$ski for $A-$bounded compact operators can be characterized by a single $A-$unit vector, which will be discussed in our next theorem.
\begin{theorem}\label{real compact single point}
	Let $\mathbb{H}$ be a real Hilbert space be such that $B_{\mathbb{H}(A)} \cap \overline{R(A)}$ is bounded with respect to $\|\cdot\|$. Let $\epsilon \in [0,1).$ Let $T,S \in  B_{A^{1/2}}(\mathbb{H}) \cap \mathbb{K(\mathbb{H})}.$ Then  $T ~\bot_{\epsilon(A)}~ S$ if and only if there exists an $A-$unit vector $x$, such that $\|Tx\|_A =\|T\|_A$ and $|\langle Tx,Sx \rangle_A| \leq \epsilon \|T\|_A \|S\|_A.$
\end{theorem}
\begin{proof}Since the sufficient part follows trivially, 
	we only prove the necessary part of the theorem. By Theorem \ref{real A bounded}, $T ~\bot_{\epsilon(A)}~ S$ if and only if
	there exists a sequence $\{x_n\} \subseteq S_{\mathbb{H}(A)} \cap \overline{R(A)} ~(\subseteq S_{\mathbb{H}(A)})$ such that
	$\lim_{n \to \infty} \|Tx_n\|_A = \|T\|_A$ and
	$\lim_{n \to \infty} |\langle Tx_n,Sx_n \rangle_A| \leq \epsilon \|T\|_A \|S\|_A.$
	As $B_{\mathbb{H}(A)} \cap \overline{R(A)}$ is  bounded with respect to $\|\cdot\|,$ clearly, $B_{\mathbb{H}(A)} \cap \overline{R(A)}$ is weakly compact with respect to $\|\cdot\|.$  Thus, the sequence $\{x_{n}\}$ has a weakly convergent subsequence. Without loss of generality we may assume that $x_{n} \rightharpoonup x$ with respect to $\|\cdot\|$ in $\mathbb{H}$, for some $x \in B_{\mathbb{H}(A)} \cap \overline{R(A)}.$ Since $T,S \in \mathbb{K}(\mathbb{H}),$ it follows that $Tx_{n} \longrightarrow Tx$ and $Sx_{n} \longrightarrow Sx$ with respect to $\|\cdot\|$ in $\mathbb{H}$. Therefore,
	$$\lim_{n \to \infty} \|Tx_{n}\|_A = \|Tx\|_A=\|T\|_A.$$ Thus, $x \in S_{\mathbb{H}(A)}.$ Also we have, $\lim_{n \to \infty} |\langle Tx_{n},Sx_{n}\rangle_A|
		= |\langle Tx,Sx\rangle_A|
		\leq  \epsilon\|T\|_A \|S\|_A.$ This completes the proof.
	\end{proof}
\begin{remark}
 Note that in Theorem \ref{real compact single point}, if we consider $\epsilon=0,$ Theorem 2.8 of \cite{ssp} is obtained. Moreover, if $\mathbb{H}$ is finite-dimensional, $A=I$ and $\epsilon=0,$ Bhatia-$\breve{S}$emrl Theorem \cite[Th.1.1]{BS} follows immediately.
\end{remark}
  In \cite{ssp}, the authors proved the following:
\begin{theorem}
	Let $\mathbb{H}$  be a Hilbert space and $\epsilon \in [0,1).$  Suppose  $B_{\mathbb{H}(A)} \cap \overline{R(A)}$ is bounded with respect to $\|\cdot\|.$  Let $T,S \in B_{A^{1/2}}(\mathbb{H}) \cap \mathbb{K}(\mathbb{H}).$ Then $T \bot_A^BS $ if and only if there exists $v \in M^T_A$ such that $Tv\bot_ASv.$
\end{theorem}
  Hence, $A$-orthogonality of $A-$bounded compact operator is determined by some $A$-unit vector $v \in M^T_A$. In our next theorem we establish that  $(\epsilon,A)-$ approximate orthogonality in the sense of Chmieli$\acute{n}$ski of $A$-bounded compact operator can  be determined by some $A$-unit vector $x \in M^T_A$ under some additional conditions.
\begin{theorem}\label{real compact tx app ortho sx}
   Let $\mathbb{H}$ be a real Hilbert space and $\epsilon \in [0,1).$ Suppose  $B_{\mathbb{H}(A)} \cap \overline{R(A)}$ is bounded with respect to $\|\cdot\|$. Let $T,S \in B_{A^{1/2}}(\mathbb{H}) \cap \mathbb{K(\mathbb{H})}$ be such that  $M_{A}^{T} \subseteq M_{A}^{S}$. Then $T \bot_{\epsilon(A)} S$ if and only if $Tx \bot_{\epsilon(A)} Sx,$ where $x \in M_{A}^{T}.$
\end{theorem}
\begin{proof}
 By Theorem \ref{real compact single point}, $T \bot_{\epsilon(A)} S$ if and only if there exits $x \in M_{A}^{T}$ such that $|\langle Tx,Sx\rangle_A| \leq \epsilon \|T\|_A \|S\|_A.$ As $M_{A}^{T} \subseteq M_{A}^{S}$, it is easy to see that, $|\langle Tx,Sx\rangle_A| \leq \epsilon \|T\|_A \|S\|_A =\epsilon \|Tx\|_A \|Sx\|_A.$ Hence  $Tx \bot_{\epsilon(A)} Sx.$\\
   Conversely let $\lambda \in \mathbb{R}.$  Therefore,
      $$\|T+\lambda S\|_A^2 \geq \|(T+\lambda S)x\|_A^2
         \geq\|Tx\|_A^2 - 2 |\lambda \langle Tx,Sx\rangle_A|
        \geq \|T\|_A^2 - 2 \epsilon\|T\|_A \|\lambda S\|_A.$$

     Hence, $T \bot_{\epsilon(A)}S.$ This completes the proof of the theorem.
\end{proof}

\section{ Symmetry of $(\epsilon,A)-$approximate orthogonality of operators}
By Theorem \ref{equality} and Proposition \ref{commutative}, it is easy to see that $x\bot_{\epsilon(A)}y$ if and only if $y\bot_{\epsilon(A)}x,$ where $x,y \in \mathbb{H}.$ But this is not necessarily true in case of  $(\epsilon,A)-$approximate orthogonality of operators  in the sense of Chmieli$\acute{n}$ski. We begin this section with an easy example to illustrate this fact.
\begin{example}\label{EXAMPLE}
	Consider $\mathbb{R}^2$ with usual inner product. Let $A(x,y)=(x,2y)$ for all $(x,y) \in \mathbb{R}^2.$ Let $T(x,y)=(2x,y)$ and $S(x,y)=(0,y)$ for all $(x,y) \in \mathbb{R}^2.$ Let $\epsilon = \frac{1}{3}.$ We show that $T\bot_{\epsilon(A)}S$ but $S\not\perp_{\epsilon(A)}T.$ Clearly, $\|T\|_A=2$ and $\|S\|_A=1$. It is easy to see that $M_{A}^{T}=\{\pm (1,0)\}$ and $M_{A}^{S}=\{\pm (0,\frac{1}{\sqrt{2}})\}.$ Further note that $|\langle T(1,0),S(1,0)\rangle_A| = 0 \leq \epsilon \|T\|_A \|S\|_A$ but $|\langle S(0,\frac{1}{\sqrt{2}}),T(0,\frac{1}{\sqrt{2}})\rangle_A| = 1 \nleq \epsilon \|T\|_A \|S\|_A.$ Hence, $T\bot_{\epsilon(A)}S$ but $S\not\perp_{\epsilon(A)}T.$
\end{example}
It is now natural to ask if $T,S \in B_{A^{1/2}}(\mathbb{H}),$ then under what conditions
 $T\bot_{\epsilon(A)}S$ implies $S\perp_{\epsilon(A)}T.$ Our next proposition gives an easy sufficient condition fot this to happen. The proof is omitted as it follows directly from Theorem \ref{real compact single point} and Theorem \ref{real compact tx app ortho sx}.

\begin{prop}\label{Mat subset Mas}
	Let $\mathbb{H}$ be a real Hilbert space and $\epsilon \in [0,1).$    Suppose  $B_{\mathbb{H}(A)} \cap \overline{R(A)}$ is bounded  with respect to $\|\cdot\|$. Let $T,S \in B_{A^{1/2}}(\mathbb{H}) \cap \mathbb{K(\mathbb{H})}$ be such that $M_{A}^{T} \subseteq M_{A}^{S}$.  If $T \bot_{\epsilon(A)} S,$  then $S \bot_{\epsilon(A)} T.$
\end{prop}

 In the following corollary we show that $T \bot_{\epsilon(A)} S \Leftrightarrow S \bot_{\epsilon(A)} T$ holds under some additional condition.
\begin{corollary}
  Let $\mathbb{H}$ be a real Hilbert space and  $\epsilon \in [0,1).$ Suppose  $B_{\mathbb{H}(A)} \cap \overline{R(A)}$ is bounded  with respect to $\|\cdot\|$. Let $T,S \in B_{A^{1/2}}(\mathbb{H}) \cap \mathbb{K(\mathbb{H})}$ be such that $M_{A}^{T} = M_{A}^{S}$.  Then $T \bot_{\epsilon(A)} S$ if and only if  $S \bot_{\epsilon(A)} T.$
\end{corollary}
\begin{proof}
  The proof follows trivially from Proposition \ref{Mat subset Mas}.
\end{proof}

In \cite{gsp}, the authors proved that in a finite-dimensional real Hilbert space, an operator is right symmetric if and only if it is an isometry. In our next theorem, we obtain the characterization of $(\epsilon,A)-$approximate right symmetric $A-$bounded operator in the sense of Chmieli$\acute{n}$ski by means of  $A$-isometry. Note that  an element $T \in B_{A^{1/2}}(\mathbb{H})$ is said to be \textit{$A$-isometry} if $M_A^{T} = S_{\mathbb{H}(A)}.$
\begin{theorem}\label{right symmetric characterization}
 Let $\mathbb{H}$ be a real Hilbert space such that $dim\overline{R(A)} < \infty.$ Let $\epsilon \in [0,1)$ and $T \in B_{A^{1/2}}(\mathbb{H})$. Then $T$ is $(\epsilon,A)-$approximate right symmetric point in the sense of Chmieli$\acute{n}$ski if and only if $T$ is  an $A$-isometry.
\end{theorem}
\begin{proof}
  First we consider $T $ to be an $A$-isometry. Let $S \in B_{A^{1/2}}(\mathbb{H})$ be such that $S \bot_{\epsilon(A)} T.$ Since,  $M_A^{T} = S_{\mathbb{H}(A)},$ by Theorem \ref{real A bounded}, it follows that $T \bot_{\epsilon(A)} S.$
      \newline
      Next we prove the necessary part of the theorem.
       Let $A_0 = A\mid_{\overline{R(A)}}.$ Let $P$ be the orthogonal projection on $\overline{R(A)}.$ Let $\tilde{T}= PT\mid_{\overline{R(A)}}.$ Suppose on the contrary that $M_A^{T} \neq S_{\mathbb{H}(A)}.$ Then, by (vi) of Proposition \ref{helpful}, $ M_{A_0}^{\tilde{T}} = M_A^{T} \cap \overline{R(A)} \neq S_{\mathbb{H}(A)} \cap \overline{R(A)}.$  Without loss of generality we assume that $\|T\|_A = \|\tilde{T}\|_{A_0} =1.$ By \cite[Th.2.4]{ssp},  $M_{A_0}^{\tilde{T}}$ is the $A_0-$unit sphere of some subspace  $H_0$ of $\overline{R(A)}.$ As $A_0$ is positive definite on $\overline{R(A)}$, it follows that $\langle ~,~\rangle_{A_0}, \|\cdot\|_{A_0}$ are inner product and norm on $\overline{R(A)},$ respectively. As $\overline{R(A)}$ is finite-dimensional, without loss of generality we assume that $\{x_1,x_2,...,x_m\}$ be an $A_0-$orthonormal basis of $H_0$. Then, $\{x_1,x_2,...,x_m,x_{m+1},...,x_n\}$ is an $A_0-$orthonormal basis of $\overline{R(A)}$.  By \cite[Th.2.11]{ssp}, it is clear that $$\|\sum_{i=1}^{m} c_i \tilde{T}x_i\|_{A_0}^2 = \|\tilde{T}\|_{A_0}^2 \|\sum_{i=1}^{m} c_i x_i\|_{A_0}^2 =\sum_{i=1}^{m} |c_i|^2.$$ As $M_{A_0}^{\tilde{T}} \neq S_{\mathbb{H}(A)} \cap \overline{R(A)},$ it follows that  $\tilde{T}(H_0) \neq \overline{R(A)}.$ So, there exists an $A_0-$unit vector $w_0 \in \overline{R(A)}$ such that $w_0 \bot_{A_0} \tilde{T}(H_0).$ Thus, for $w_0, \tilde{T}x_{m+1},$ by ($i$) of Proposition \ref{x+,x-}, either $$\|w_0 +\lambda \tilde{T}x_{m+1}\|_{A_0} \geq 1 ~\forall\lambda \geq 0~\text{or}~\|w_0 +\lambda \tilde{T}x_{m+1}\|_{A_0} \geq 1 ~\forall\lambda \leq 0.$$ Let $\|w_0 +\lambda \tilde{T}x_{m+1}\|_{A_0} \geq 1 ~\forall\lambda \geq 0$. Consider,
\begin{eqnarray*}
       \tilde{S} &:& \overline{R(A)} \longrightarrow \overline{R(A)}\\
        \tilde{S}x_i &=& -\tilde{T}x_i ~\text{for all}~ i\in\{1,2,...,m\}\\
         \tilde{S}x_i &=& w_0  ~\text{for}~i=m+1\\
         \tilde{S}x_i &=& 0~\text{for all}~ i\in\{m+2,m+3,...,n\}.
       \end{eqnarray*}
       Now we show that $\tilde{S} \bot_{\epsilon(A_0)} \tilde{T}.$ Let $z = \sum_{i=1}^{n} c_i x_i \in S_{\mathbb{H}(A)} \cap  \overline{R(A)}.$ It is easy to see that $\sum_{i=1}^{n} |c_i|^2 =1.$ Clearly, $\|\tilde{S}z\|_{A_0} ^2 = \|\sum_{i=1}^{m} c_i \tilde{T}x_i\|_{A_0}^2 + c_{m+1}^2 = \sum_{i=1}^{m} |c_i|^2  + c_{m+1}^2 \leq \sum_{i=1}^{n} |c_i|^2 =1.$ We also see that $\|\tilde{S}x_{m+1}\|_{A_0} =1.$ Thus, $\|\tilde{S}\|_{A_0} =1.$ Further note that $x_{i} \in M_{A_0}^{\tilde{S}} ~\text{for all} ~ i \in \{1,2,...,m+1\}.$ Now for $\lambda \geq 0,$ $$\|\tilde{S} + \lambda \tilde{T}\|_{A_0} \geq \|(\tilde{S} + \lambda \tilde{T})x_{m+1}\|_{A_0} =\|w_0 + \lambda \tilde{T}x_{m+1}\|_{A_0} \geq 1.$$ For $\lambda \leq 0,$ $$\|\tilde{S} + \lambda \tilde{T}\|_{A_0} \geq \|(\tilde{S} + \lambda \tilde{T})x_{m}\|_{A_0} =\| -\tilde{T}x_m + \lambda \tilde{T}x_{m}\|_{A_0} \geq 1.$$ Therefore, $\tilde{S} \bot_{A_0}^{B} \tilde{T}$ and so $\tilde{S} \bot_{\epsilon(A_0)} \tilde{T}.$ Next we show that $\tilde{T} \not\perp_{\epsilon(A_0)} \tilde{S}.$ It is easy to see that for any $x \in M_{A_0}^{\tilde{T}}$, $|\langle \tilde{T}x,\tilde{S}x\rangle_{A_0}| =\|\tilde{T}x\|_{A_0}^2 =1 > \epsilon.$ As $\overline{R(A)}$ is finite-dimensional, $\tilde{T},\tilde{S}$ are compact operators with respect to $\|\cdot\|$ on $\overline{R(A)}.$ Hence, by Theorem \ref{real compact single point}, $\tilde{T} \not\perp_{\epsilon(A_0)} \tilde{S}.$ Consider  $U : ~\mathbb{H}~ \longrightarrow ~\mathbb{H}$ by  $$U(x) = \tilde{S}x , x \in \overline{R(A)} ~\text{and}~
         U(x)=0, x \in N(A).$$
       It is easy to see that $\tilde{U} = \tilde{S}$. Thus, by Remark \ref{t implies t tilda}, $U \bot_{\epsilon(A)} T$ but $T \not\perp_{\epsilon(A)} U.$ This completes the proof of the theorem.
       \end{proof}
       As a consequence of Theorem \ref{right symmetric characterization}, we can immediately obtain the following:
       \begin{corollary}\label{T is isometry}
          Let $\mathbb{H}$ be a finite-dimensional real Hilbert space. Let $\epsilon \in [0,1)$ and $T \in \mathbb{L(\mathbb{H})}$. Then $T$ is approximate right symmetric if and only if $T$ is an isometry.
         \end{corollary}

       \begin{remark}\label{2,7,4.4} By Theorem \ref{right symmetric characterization}, it follows that if $\mathbb{H}$ is finite-dimensional and $T \in B_{A^{1/2}}(\mathbb{H}),$ then $T$ is $(\epsilon,A)-$approximate right symmetric point in the sense of Chmieli$\acute{n}$ski if and only if $T$ is  $A-$isometry. Also
         note that in Corollary \ref{T is isometry}, if we consider  $\epsilon =0$, we obtain the characterization of right symmetric point in $\mathbb{L(\mathbb{H})}$. Hence, Theorem \ref{right symmetric characterization} is a  generalization of  \cite[Th.2.7]{gsp} and  \cite[Th.4.4]{turn} in finite-dimensional case.
       \end{remark}
Now we obtain the characterization of $(\epsilon,A)-$approximate right symmetric point in the sense of Chmieli$\acute{n}$ski for $A$-bounded compact operators in infinite-dimensional Hilbert space setting.
       \begin{theorem}\label{right symmetry infinite compact}
         Let $\mathbb{H}$ be an infinite-dimensional separable real Hilbert space. Suppose  $B_{\mathbb{H}(A)} \cap \overline{R(A)}$ is bounded with respect to $\|\cdot\|$ and $dim \overline{R(A)}=\infty$. Let $\epsilon \in [0,1)$ and $T \in B_{A^{1/2}}(\mathbb{H}) \cap \mathbb{K(\mathbb{H})}$. Then $T$ is $(\epsilon,A)-$approximate right symmetric point in the sense of Chmieli$\acute{n}$ski if and only if $\|T\|_A =0.$
       \end{theorem}
       \begin{proof}We only prove the necessary part of the theorem, as the sufficient part follows trivially. Suppose on the contrary that $\|T\|_A \neq 0.$ Without loss of generality we assume that $\|T\|_A =1.$  Let $A_0 = A\mid_{\overline{R(A)}}.$ Let $P$ be the orthogonal projection on $\overline{R(A)}.$ Let $\tilde{T}= PT\mid_{\overline{R(A)}}.$ Then, $\|\tilde{T}\|_{A_0} =1.$ As $A_0$ is positive definite on $\overline{R(A)}$, $\|\cdot\|_{A_0}$ and $\langle~.~\rangle_{A_0}$ are norm and inner product on $\overline{R(A)},$ respectively. It is easy to see that $\|\cdot\|$ and $\|\cdot\|_{A_0}$ are equivalent norms on $\overline{R(A)}$(\cite{ssp}).
        Next we show that, if $T \in B_{A^{1/2}}(\mathbb{H}) \cap \mathbb{K(\mathbb{H})}$ and $A$ is positive definite on $\overline{R(A)}$, then $M_A^T \cap \overline{R(A)} \neq S_{\mathbb{H}(A)} \cap \overline{R(A)}.$
         Clearly, $M_A^T \cap \overline{R(A)}$ is compact with respect to $\|\cdot\|$ (\cite{ssp}). Also note that, as $A$ is positive definite on $\overline{R(A)}$, $B_{\mathbb{H}(A)} \cap \overline{R(A)}$ is the unit ball of $\overline{R(A)}$ with respect to $\|\cdot\|_{A_0}.$ As $\overline{R(A)}$ is infinite-dimensional, $B_{\mathbb{H}(A)} \cap \overline{R(A)}$
         is not compact with respect to $\|\cdot\|_{A_0}.$ As $\|\cdot\|$ and $\|\cdot\|_{A_0}$ are equivalent on $\overline{R(A)},$ it follows that $B_{\mathbb{H}(A)} \cap \overline{R(A)}$ is not compact with respect to $\|\cdot\|.$ Therefore, $M_A^T \cap \overline{R(A)} \neq S_{\mathbb{H}(A)} \cap \overline{R(A)}.$
         Now we are ready to construct $\tilde{S} \in B_{{A_0}^{1/2}}(\overline{R(A)})$ such that $\tilde{S} \bot_{\epsilon(A_0)} \tilde{T}$ but $\tilde{T} \not \perp_{\epsilon(A_0)} \tilde{S}.$\\ By \cite[Th.2.4]{ssp}, $M_A^T \cap \overline{R(A)}$ is $A$-unit sphere of some subspace $H_0$ of $\overline{R(A)}.$ Therefore, $M_{A_0}^{\tilde{T}}$ is $A_0-$unit sphere of some subspace $H_0.$ Let $\{x_{\alpha}, \alpha \in \Lambda_1\}$ be an $A_0-$orthonormal basis of $H_0.$ Extend this to a $A_0-$orthonormal basis $\{x_{\alpha}, y_{\beta}, \alpha \in \Lambda_1,~\beta \in \Lambda_2\}$ of $\overline{R(A)}.$ Clearly, $\tilde{T}(H_0) \neq \overline{R(A)},$ otherwise $\overline{\tilde{T}(B_{H_0})} = \overline{B_{\mathbb{H}(A)} \cap \overline{R(A)}}= B_{\mathbb{H}(A)}\cap \overline{R(A)}$ and as $\tilde{T}$ is compact on $\overline{R(A)},$ it follows that  $B_{\mathbb{H}(A)} \cap \overline{R(A)}$ is compact with respect to $\|\cdot\|$, a contradiction. Therefore, there exists $w_0 \in S_{\mathbb{H}(A)} \cap \overline{R(A)}$ such that $w_0 ~\bot_{A_0} ~\tilde{T}(H_0).$ Let $\beta_0 \in \Lambda_2.$ Hence, for $w_0, \tilde{T}y_{\beta_0}$, in view of ($i$) of Proposition \ref{x+,x-}, either $\|w_0 + \lambda \tilde{T}y_{\beta_0}\|_{A_0} \geq 1$ for all $\lambda \geq 0$ or $\|w_0 + \lambda \tilde{T}y_{\beta_0}\|_{A_0} \geq 1$ for all $\lambda \leq 0.$ Without loss of generality we assume that $\|w_0 + \lambda \tilde{T}y_{\beta_0}\|_{A_0} \geq 1$ for all $\lambda \geq 0$. Define a map $\tilde{S} ~:~ \overline{R(A)}\longrightarrow ~\overline{R(A)}$ by
         \begin{eqnarray*}
           \tilde{S} x_{\alpha} &=& -\tilde{T} x_{\alpha},~ \alpha \in \Lambda_1, \\
           \tilde{S} y_{\beta} &=& w_0,  ~\beta = \beta_0,\\
           \tilde{S} y_{\beta} &=& 0, ~\beta \neq \beta_0.
         \end{eqnarray*}
         Let $z = \sum c_{\alpha_i}x_{\alpha_i} + \sum d_{\beta_i}y_{\beta_i} \in S_{\mathbb{H}(A)} \cap \overline{R(A)}.$ As $\|\tilde{T}x_{\alpha_i}\|_{A_0} = \|\tilde{T}\|_{A_0} =1,$ for all ${\alpha_i} \in \Lambda_1,$ applying  \cite[Th.2.4]{ssp} and by the fact $w_0 \bot_{A_0} \tilde{T}(H_0)$, it follows that $$\|\tilde{S} z\|_{A_0}^2 = \|\sum c_{\alpha_i}\tilde{T}x_{\alpha_i}\|_{A_0}^2 + d_{\beta_0}^2 = \sum c_{\alpha_i}^2 + d_{\beta_0}^2 \leq 1.$$ Thus, $\|\tilde{S} \|_{A_0} =1.$
          Now for any $\lambda \geq 0,$ we have $$\|\tilde{S}  +\lambda \tilde{T}\|_{A_0} \geq \|(\tilde{S}  +\lambda \tilde{T})y_{\beta_0}\|_{A_0} \geq 1.$$ Similarly for any $\lambda \leq 0,$ we obtain $$\|\tilde{S}  +\lambda \tilde{T} \|_{A_0} \geq \|(\tilde{S}  +\lambda \tilde{T})x_{\alpha}\|_{A_0} \geq 1.$$ Hence, $\tilde{S}  \bot_{\epsilon(A_0)} \tilde{T}.$ \\
         Next we show that $\tilde{T}  \not\perp_{\epsilon(A_0)} \tilde{S}.$  Let $\{x_n\} \subseteq S_{\mathbb{H}(A)} \cap \overline{R(A)}$ such that $\|\tilde{T}x_n\|_{A_0} \longrightarrow \|\tilde{T}\|_{A_0}.$ As $B_{\mathbb{H}(A)} \cap \overline{R(A)}$ is closed, convex and bounded with respect to $\|\cdot\|,$  $B_{\mathbb{H}(A)} \cap \overline{R(A)}$ is weakly compact with respect to $\|\cdot\|.$ Thus, $x_n \rightharpoonup x$ with respect to $\|\cdot\|$ for some $x \in B_{\mathbb{H}(A)} \cap \overline{R(A)}.$  As $\tilde{T}$ is a compact operator on $\overline{R(A)},$ it follows that $\tilde{T}x_n \longrightarrow \tilde{T}x$  with respect to $\|\cdot\|$  and $\|\tilde{T}x\|_{A_0} = \|\tilde{T}\|_{A_0}.$ Hence, $x \in M_{A_0}^{\tilde{T}}$. As $\tilde{S} \in B_{{A_0}^{1/2}}(\overline{R(A)}),$ it is easy to see that $\tilde{S}x_n \rightharpoonup \tilde{S}x$  with respect to $\|\cdot\|.$ As $\langle~,~\rangle_{A_0}$ is inner product on $\overline{R(A)},$  $\langle \tilde{T}x_n,\tilde{S}x_n\rangle_{A_0} \longrightarrow \langle \tilde{T}x,\tilde{S}x\rangle_{A_0}.$ But for any $x \in M_{A_0}^{\tilde{T}},$ we have $|\langle \tilde{T}x,\tilde{S}x\rangle_{A_0}| = \|\tilde{T}x\|_{A_0}^2 = 1 > \epsilon.$ Thus, $\tilde{T} \not\perp_{\epsilon(A_0)} \tilde{S}.$ Consider $U ~: ~\mathbb{H} \longrightarrow ~\mathbb{H}$ by $$Ux = \tilde{S}x, ~ x \in \overline{R(A)}~ \text{and}~ Ux = 0, ~ x \in N(A).$$ Clearly, $\tilde{U} = \tilde{S}.$ Therefore, by Remark \ref{t implies t tilda}, $U \perp_{\epsilon(A)} T$ but $T \not\perp_{\epsilon(A)} U.$
       \end{proof}
   \begin{remark}\label{2,8}
   	In \cite{gsp}, the authors proved that in infinite-dimensional Hilbert space setting, any compact operator is right symmetric if and only if it is zero operator.
   Thus our theorem generalizes \cite[Th.2.8]{gsp}.
   \end{remark}
      As a consequence of  Theorem \ref{right symmetry infinite compact}, we can immediately establish the following corollary, the proof of which is omitted as it is now trivial.
       \begin{corollary}
          Let $\mathbb{H}$ be an infinite-dimensional separable real Hilbert space. Let $\epsilon \in [0,1)$ and $T \in \mathbb{K(\mathbb{H})}$. Then $T$ is approximate right symmetric if and only if $T$ is the zero operator.
       \end{corollary}

       Next we characterize $(\epsilon,A)-$approximate left symmetry in the sense of Chmieli$\acute{n}$ski for $A-$bounded compact operator in $\mathbb{H}$.  To do so we need  the following lemmas.
       \begin{lemma}\label{left symmetry lemma}
         Let $\epsilon_{1}, \epsilon \in [0,1)$ and $\epsilon_{1} > \epsilon.$ Then there exists  \newline $ \ a \in (\epsilon\epsilon_{1} - \sqrt{1-\epsilon^2}\sqrt{1-\epsilon_{1}^2}, \ \epsilon\epsilon_{1} + \sqrt{1-\epsilon^2}\sqrt{1-\epsilon_{1}^2})$ such that
         \begin{itemize}
           \item [($i$)]   $a\epsilon_{1} > \epsilon.$
           \item [($ii$)]$\frac{a\epsilon_{1}- \epsilon}{\sqrt{1-\epsilon_{1}^2}b} < 1,$ where $a^2+b^2 =1$ and $b > 0.$
         \end{itemize}
       \end{lemma}
       \begin{proof}
          Choose $a = \epsilon\epsilon_{1} + (1-2t)\sqrt{1-\epsilon^2}\sqrt{1-\epsilon_{1}^2},$ where $0 < t < \frac{1}{2}(1 - \frac{\epsilon \sqrt{1-\epsilon_{1}^2}}{\epsilon_{1} \sqrt{1-\epsilon^2}}).$ \\
          ($i$) As $\epsilon_{1}  > \epsilon$ and $\epsilon_{1}, \epsilon \in [0,1),$ it is easy to see that $ 0 \leq \frac{\epsilon \sqrt{1-\epsilon_{1}^2}}{\epsilon_{1} \sqrt{1-\epsilon^2}} < 1.$ As $t \in (0,1),$ it follows that $a \in (\epsilon\epsilon_{1} - \sqrt{1-\epsilon^2}\sqrt{1-\epsilon_{1}^2},\epsilon\epsilon_{1} + \sqrt{1-\epsilon^2}\sqrt{1-\epsilon_{1}^2}) \subset (-1,1).$  Therefore,
        $$a\epsilon_{1} >  \epsilon\epsilon_{1}^2 + \frac{\epsilon \sqrt{1-\epsilon_{1}^2}}{\epsilon_{1} \sqrt{1-\epsilon^2}}\epsilon_{1}\sqrt{1-\epsilon^2}\sqrt{1-\epsilon_{1}^2}
            >  \epsilon\epsilon_{1}^2 + \epsilon(1-\epsilon_{1}^2)\\
            = \epsilon.$$

       ($ii$) Suppose on the contrary that $\frac{a\epsilon_{1}- \epsilon}{\sqrt{1-\epsilon_{1}^2}b} \geq 1.$ As $a\epsilon_{1}- \epsilon, \sqrt{1-\epsilon_{1}^2}b > 0,$ we have
         \begin{eqnarray*}
           a\epsilon_{1}- \epsilon &\geq& \sqrt{1-\epsilon_{1}^2}b \\
           a^2\epsilon_{1}^2 - 2a\epsilon\epsilon_{1} + \epsilon^2 &\geq&  (1-\epsilon_{1}^2)b^2 \\
           \epsilon_{1}^2 - 2a\epsilon\epsilon_{1} + \epsilon^2 - b^2 &\geq & 0\\
           a^2 - 2a\epsilon\epsilon_{1} + \epsilon_{1}^2 +\epsilon^2 -1 &\geq& 0.
         \end{eqnarray*}
         But then $a \geq \epsilon\epsilon_{1} + \sqrt{1-\epsilon^2}\sqrt{1-\epsilon_{1}^2}$ or $a \leq \epsilon\epsilon_{1} - \sqrt{1-\epsilon^2}\sqrt{1-\epsilon_{1}^2},$ a contradiction.
       \end{proof}
   \begin{lemma}\label{useful2}
   	Let $\mathbb{H}$ be an infinite-dimensional  real Hilbert space. Let  $B_{\mathbb{H}(A)} \cap \overline{R(A)}$ be bounded with respect to $\|\cdot\|.$ Let $\epsilon \in [0,1)$ and $T \in B_{A^{1/2}}(\mathbb{H}) \cap \mathbb{K(\mathbb{H})}$. If $M_{A}^{T} \cap \overline{R(A)}$ contains more than one pair of points, then $T$ cannot be $(\epsilon,A)-$approximate left symmetric point in the sense of Chmieli$\acute{n}$ski.
   \end{lemma}
\begin{proof}
	Let $A_0 = A\mid_{\overline{R(A)}}.$ Let $P$ be the orthogonal projection on $\overline{R(A)}.$ Let $\tilde{T}= PT\mid_{\overline{R(A)}}.$ Since, $M_{A_0}^{\tilde{T}}$ is $A_0$ unit sphere of some subspace say $H_0$ of $\overline{R(A)},$ it follows that there exist $z_1,z_2 \in S_{\mathbb{H}(A)} \cap \overline{R(A)}$ such that $z_1 \bot_{A_0} z_2$ and $\|\tilde{T}z_1\|_{A_0}=\|\tilde{T}z_2\|_{A_0}=\|\tilde{T}\|_{A_0}.$ Without loss of generality we assume that $\|\tilde{T}\|_{A_0}=1.$ Clearly, $\overline{R(A)} = \langle\{z_2\}\rangle \oplus H,$ where $H  =\langle\{z_2\}\rangle^{\bot_{A_0}}.$ We define a map $\tilde{S} ~:~ \overline{R(A)}~\longrightarrow ~ \overline{R(A)}$ by $\tilde{S}(H) = \theta,$ and $\tilde{S}z_2 = \tilde{T}z_2.$ Clearly, $\tilde{S}$ is compact with respect to $\|\cdot\|$ and $M_{A_0}^{\tilde{S}} = \{\pm z_2\}$ and $\|\tilde{S}\|_{A_0} =1.$ As $\tilde{S}z_1 = \theta,$ it follows that $|\langle \tilde{T}z_1,\tilde{S}z_1\rangle_{A_0}| = 0$ and so $\tilde{T} \perp_{\epsilon(A_0)} \tilde{S}$. But $|\langle \tilde{S}z_2,\tilde{T}z_2\rangle_{A_0}| = \|\tilde{T}z_2\|_{A_0}^2= 1 > \epsilon.$ Hence,  $\tilde{S} \not\perp_{\epsilon(A_0)} \tilde{T}.$
	This completes the proof.
	\end{proof}
       Now we are ready to characterize $(\epsilon,A)-$approximate left symmetric point in the sense of Chmieli$\acute{n}$ski for $A$-bounded compact operator in infinite-dimensional Hilbert spaces.
       \begin{theorem}\label{left symmetry infinite}
          Let $\mathbb{H}$ be an infinite-dimensional separable real Hilbert space. Suppose $B_{\mathbb{H}(A)} \cap \overline{R(A)}$ is bounded with respect to $\|\cdot\|.$ Let $\epsilon \in [0,1)$ and $T \in B_{A^{1/2}}(\mathbb{H}) \cap \mathbb{K(\mathbb{H})}$. Then $T$ is $(\epsilon,A)-$approximate left symmetric point in the sense of Chmieli$\acute{n}$ski if and only if $\|T\|_A =0.$
       \end{theorem}
       \begin{proof}Sufficient part of the theorem follows trivially. We only prove the necessary part.\\
         Suppose on the contrary that $\|T\|_A \neq 0.$ Without loss of generality, we assume that $\|T\|_A =1.$ Let $A_0 = A\mid_{\overline{R(A)}}.$ Let $P$ be the orthogonal projection on $\overline{R(A)}.$ Let $\tilde{T}= PT\mid_{\overline{R(A)}}.$ As $\|T\|_A =1,$ therefore $\|\tilde{T}\|_{A_0} =1.$ Clearly, $\tilde{T}$ is compact with respect to $\|\cdot\|.$ Since  $\|\cdot\|$ and $\|\cdot\|_{A_0}$ are equivalent on $\overline{R(A)}$ and $ B_{\mathbb{H}(A)} \cap \overline{R(A)}$ is weakly compact with respect to $\|\cdot\|,$ there exists $x \in S_{\mathbb{H}(A)} \cap \overline{R(A)}$ such that $\|\tilde{T}x\|_{A_0}=\|\tilde{T}\|_{A_0} =1.$ By virtue of Lemma \ref{useful2}, we assume $M_{A_0}^{\tilde{T}}=\{\pm x\}.$ If we can construct $\tilde{S}$ on $\overline{R(A)}$ such that $\tilde{T} \perp_{\epsilon(A_0)} \tilde{S}$ but $\tilde{S} \not\perp_{\epsilon(A_0)} \tilde{T},$ then we are done. Clearly, $\overline{R(A)} = \langle\{x\}\rangle \oplus H_1,$ where $H_1 =\langle\{x\}\rangle^{\bot_{A_0}}.$ Let $\{x_1,x_2,...\}$ be an $A_0$ orthonormal basis of $H_1.$ Therefore, $\{x,x_1,x_2,...\}$ is an $A_0$ orthonormal basis of $\overline{R(A)}.$\\
         \underline{Case ($I$)}:
           $\tilde{T}(H_1)=\{\theta\}.$\\
           Consider $z_1 = \epsilon_1 x + \sqrt{1 - \epsilon_1^2} x_1$ and $z_2 = - \sqrt{1 - \epsilon_1^2} x + \epsilon_1 x_1,$ where $\epsilon_1 \in [0,1)$ and $\epsilon_1 >\epsilon.$
         Clearly, $\{z_1,z_2,x_2,...\}$ is an  $A_0$ orthonormal basis of $\overline{R(A)}.$ It is immediate to see that $\overline{R(A)} = \langle\{z_1,z_2\}\rangle \oplus H_2,$ where $H_2 =\langle\{z_1,z_2\}\rangle^{\bot_{A_0}}.$ We define a map $\tilde{S} ~:~ \overline{R(A)}~\longrightarrow ~ \overline{R(A)}$ by
         $$ \tilde{S}(H_2) =\theta,
            \tilde{S}z_1 = a\tilde{T}x + b w,
            \tilde{S}z_2 =\alpha(b \tilde{T}x - a w);$$
           where the choice of $a$ is the same, described in Lemma \ref{left symmetry lemma} (i.e. $a = \epsilon\epsilon_{1} + (1-2t)\sqrt{1-\epsilon^2}\sqrt{1-\epsilon_{1}^2},$ where $0 < t < \frac{1}{2}(1 - \frac{\epsilon \sqrt{1-\epsilon_{1}^2}}{\epsilon_{1} \sqrt{1-\epsilon^2}})$) with $a^2 +b ^2 =1$, $b>0$ and $\frac{a\epsilon_{1}- \epsilon}{\sqrt{1-\epsilon_{1}^2}b} < \alpha < \min\{1, \frac{a\epsilon_{1}+ \epsilon}{\sqrt{1-\epsilon_{1}^2}b}\}$ and $w \bot_{A_0} \tilde{T}x,$ $w \in S_{\mathbb{H}(A)} \cap \overline{R(A)}.$ Since, $\tilde{S}$ is of finite rank and  $\|\cdot\|$ and $\|\cdot\|_{A_0}$ are equivalent on $\overline{R(A)}$,  $\tilde{S}$ is compact with respect to $\|\cdot\|$ on $\overline{R(A)}.$ Let  $z \in S_{\mathbb{H}(A)} \cap \overline{R(A)}.$ Hence, $z = c_1 z_1 + c_2 z_2 + cy,$ where $c_1,c_2,c \in \mathbb{R}$ and $y \in H_2.$ Therefore, $\tilde{S}z=(c_1 a+ c_2 \alpha b) \tilde{T}x + (c_1 b- c_2 \alpha a)w.$ Thus, $$\|\tilde{S}z\|_{A_0}^2 = (c_1 a+ c_2 \alpha b)^2 + (c_1 b- c_2 \alpha a)^2 = c_1^2  +\alpha^2 c_2^2 \leq c_1^2+c_2^2 \leq 1. $$ It is easy to see that $\|\tilde{S}z_1\|_{A_0} = \sqrt{a^2 +b^2} =1.$ Thus, $M_{A_0}^{\tilde{S}} = \{\pm z_1\}$ and $\|\tilde{S}\|_{A_0} =1.$ Now we show that $\tilde{T} \perp_{\epsilon(A_0)} \tilde{S}$ but $\tilde{S} \not\perp_{\epsilon(A_0)} \tilde{T}.$ By our construction, $x = \epsilon_1 z_1 - \sqrt{1 - \epsilon_1^2} z_2.$ So,
           $\langle \tilde{T}x, \tilde{S}x\rangle_{A_0} = \langle \tilde{T}x, \epsilon_1\tilde{S}z_1 - \sqrt{1-\epsilon_1^2}\tilde{S}z_2\rangle_{A_0}
               = \epsilon_1 a - \sqrt{1-\epsilon_1^2}\alpha b.$

              By our choice of $\alpha,$ it is easy to see that $- \epsilon < \epsilon_1 a - \sqrt{1-\epsilon_1^2}\alpha b < \epsilon.$ Hence, $|\langle \tilde{T}x, \tilde{S}x\rangle_{A_0}| \leq \epsilon = \epsilon \|\tilde{T}\|_{A_0}\|\tilde{S}\|_{A_0}$ and so $\tilde{T} \perp_{\epsilon(A_0)} \tilde{S}$. Again,
              $ |\langle \tilde{S}z_1, \tilde{T}z_1\rangle_{A_0}| = |\langle a\tilde{T}x +b w, \tilde{T}z_1 \rangle_{A_0}|
                 = |\langle a\tilde{T}x +b w, \epsilon_1\tilde{T}x \rangle_{A_0}|
                 = |a \epsilon_1|
                 > \epsilon.$
             Hence, $\tilde{S} \not\perp_{\epsilon(A_0)} \tilde{T}.$\\
                \underline{Case ($II$)}:
                 $\tilde{T}(H_1) \neq \{\theta\}.$\\ Thus, $\tilde{T}x_k \neq \theta,$ for some basis element $x_k \in H_1.$ Then,  $\tilde{T}x_k = \beta w,$  where $w \in S_{\mathbb{H}(A)} \cap \overline{R(A)}$ such that $w \bot_{A_0} \tilde{T}x$ and $\beta \in \mathbb{R}.$ Without loss of generality we may assume that $\beta \geq 0.$ Let $z_1 = \epsilon_1 x + \sqrt{1 - \epsilon_1^2} x_k$ and
           $z_2 = - \sqrt{1 - \epsilon_1^2} x + \epsilon_1 x_k,$ where $\epsilon_1 \in [0,1)$ and $\epsilon_1 >\epsilon.$
         Clearly, $\{z_1,z_2,x_1,x_2,...x_{k-1},x_{k+1},...\}$ is an  $A_0$ orthonormal basis of $\overline{R(A)}.$ It is immediate to see that $\overline{R(A)} = \langle\{z_1,z_2\}\rangle \oplus H_3,$ where $H_3 =\langle\{z_1,z_2\}\rangle^{\bot_{A_0}}.$ We define a map $\tilde{S} ~:~ \overline{R(A)}~\longrightarrow ~ \overline{R(A)}$ by
           $$\tilde{S}(H_3) = \theta,
            \tilde{S}z_1 = a\tilde{T}x + b w,
            \tilde{S}z_2 = \alpha(b \tilde{T}x - a w);$$

            where the choice of $a$ is the same, described in Lemma \ref{left symmetry lemma} and $a^2 +b ^2 =1$, $b>0$ and $\frac{a\epsilon_{1}- \epsilon}{\sqrt{1-\epsilon_{1}^2}b} < \alpha < \min\{1, \frac{a\epsilon_{1}+ \epsilon}{\sqrt{1-\epsilon_{1}^2}b}\}.$  Since $\tilde{S}$ is compact with respect to $\|\cdot\|$ on $\overline{R(A)},$ Proceeding in the same manner as in case ($I$), we obtain $M_{A_0}^{\tilde{S}}=\{\pm z_1\}$ and $\|\tilde{S}\|_{A_0} =1.$ Also, we have
           $\langle \tilde{T}x, \tilde{S}x\rangle_{A_0} = \langle \tilde{T}x, \epsilon_1\tilde{S}z_1 - \sqrt{1-\epsilon_1^2}\tilde{S}z_2\rangle_{A_0}
               = \epsilon_1 a - \sqrt{1-\epsilon_1^2}\alpha b.$ By our choice of $\alpha,$ it is easy to see that $- \epsilon < \epsilon_1 a - \sqrt{1-\epsilon_1^2}\alpha b < \epsilon.$ Hence $|\langle \tilde{T}x, \tilde{S}x\rangle_{A_0}| \leq \epsilon = \epsilon \|\tilde{T}\|_{A_0}\|\tilde{S}\|_{A_0}.$ Thus $\tilde{T} \perp_{\epsilon(A_0)} \tilde{S}$. Again, $|\langle \tilde{S}z_1, \tilde{T}z_1\rangle_{A_0}| = |\langle a\tilde{T}x +b w, \tilde{T}z_1 \rangle_{A_0}|
                 = |a \epsilon_1 + b \sqrt{1- \epsilon_1^2} \beta|
                 = a \epsilon_1 + b \sqrt{1- \epsilon_1^2} \beta
                 > \epsilon.$ Hence, $\tilde{S} \not\perp_{\epsilon(A_0)} \tilde{T}.$
            This completes the proof of the theorem.
       \end{proof}
   The characterization of $(\epsilon,A)-$approximate left symmetric point in the sense of Chmieli$\acute{n}$ski for $A-$bounded operators in a finite-dimensional real Hilbert space can be obtained in the same manner. Therefore, we omit the proof.
 \begin{corollary}\label{left syymetry finite}
    Let $\mathbb{H}$ be a finite-dimensional real Hilbert space. Let $\epsilon \in [0,1)$ and $T \in B_{A^{1/2}}(\mathbb{H})$. Then $T$ is $(\epsilon,A)-$approximate left symmetric point in the sense of Chmieli$\acute{n}$ski if and only if $\|T\|_A =0.$
 \end{corollary}
 Now we can easily prove the following:
 \begin{theorem}
   Let $\mathbb{H}$ be a finite-dimensional real Hilbert space. Let $\epsilon \in [0,1)$ and $T \in \mathbb{L(\mathbb{H})}$. Then the following conditions are equivalent:
          \begin{itemize}
            \item [($i$)] $T$ is left symmetric.
            \item [($ii$)] $T$ is approximate left symmetric.
            \item [($ii$)] $T$ is zero operator.
          \end{itemize}
        \end{theorem}
    We end this section with the following closing remark:
       \begin{remark}
         Note that in Theorem \ref{left symmetry infinite}, if we put $A=I$, we obtain the characterization of approximate left symmetric point in $\mathbb{L(\mathbb{H})}$. Moreover, if we put $A=I$ and $\epsilon =0,$ we obtain the characterization of left symmetric point in $\mathbb{L(\mathbb{H})}.$ Hence, Theorem \ref{left symmetry infinite} generalizes  \cite[Th.2.10]{gsp} and  \cite[Th.3.3]{turn} for compact operators on real Hilbert space.
       \end{remark}

\end{document}